\documentclass[11pt,usenames,dvipsnames]{amsart}
\usepackage{a4wide}
\usepackage{amssymb}
\usepackage{amsmath}
\usepackage{amsthm}
\usepackage{mathrsfs}
\usepackage{yfonts}
\usepackage{euscript}
\usepackage{tikz-cd}
\usepackage{upgreek}
\usepackage[all]{xypic}
\usepackage{xcolor} %{\color{Maroon/ForestGreen}  }
\usepackage{comment}
\usepackage{xr-hyper}

\usepackage{enotez, translations}
\usepackage[unicode,psdextra,dvipdfmx]{hyperref} %dvipdfmx deleted by Otmar
\usepackage{bookmark}
\usepackage[hyperpageref ]{backref}
\usepackage[only,llbracket,rrbracket]{stmaryrd}
\usepackage{todonotes}

\newcommand{\com}[1]{{#1}}                % Kommentarebefehl
\renewcommand{\com}[1]{\relax}    % Kommentare ausblenden
\newcommand{\Footnote}[1]{\footnote{{\color{Sepia} #1}}}  % Fussnoten nur f\"{u}r Peter und Otmar farbig in Sepia
\renewcommand{\Footnote}[2][]{\relax}  % nur Fussnoten f\"{u}r Peter und Otmar ausblenden
\theoremstyle{plain}

\newcommand{\id}{\operatorname{id}}

\newcommand{\Hom}{\operatorname{Hom}}

\newcommand{\Fil}{\operatorname{Fil}}

\newcommand{\Rep}{\operatorname{Rep}}

\newcommand{\Gal}{\operatorname{Gal}}

 \newcommand{\CC}{\mathbb{C}}

 \newcommand{\QQ}{\mathbb{Q}}
 
 \newcommand{\ZZ}{\mathbb{Z}}

 \newcommand{\cO}{\mathcal{O}}

\renewcommand{\cR}{\mathcal{R}}

 \newcommand{\Qp}{\QQ_p}

 \newcommand{\Cp}{\CC_p}

 \newcommand{\be}{\begin{equation}}
\newcommand{\ee}{\end{equation}}

\newcommand{\fpbar}{\ifmmode {\overline{\mathbb{F}_p}}\else$\mathbb{F}_p$\ \fi}

\newcommand{\fp}{\ifmmode {\mathbb{F}_p}\else$\mathbb{F}_p$\ \fi}
\newcommand{\zp}{\ifmmode \mathbb{Z}_p\else$\mathbb{Z}_p$\ \fi}
\newcommand{\zpur}{\ifmmode \widehat{\zp^{ur}}\else $\widehat{\zp^{ur}}$\ \fi}

 \newcommand{\TLT}{T_{\pi}}

   \newcommand{\cyc}{\mathrm{cyc}}

 \newcommand{\cris}{\mathrm{cris}}
 \newcommand{\dR}{\mathrm{dR}}
 \newcommand{\Iw}{\mathrm{Iw}}
 
  \newcommand{\LT}{\mathrm{LT}}
  \newcommand{\Tate}{\mathrm{Tate}}

\renewcommand{\H}{\mathrm{H}}

\newcommand{\La}{\ifmmode\Lambda\else$\Lambda$\fi}

\newcommand{\q}{\ifmmode {\mathbb Q}\else${\mathbb Q}$\ \fi}
\newcommand{\qp}{\ifmmode {\mathbb Q}_p\else${\mathbb Q}_p$\ \fi}

\newcommand{\Q}{\ifmmode {\mathbb Q}\else${\mathbb Q}$\ \fi}
\newcommand{\ql}{\ifmmode {{\mathbb Q}_l}\else${\mathbb Q}_l$\ \fi}

\newtheorem{theorem}{Theorem}[section]

\newtheorem{lemma}[theorem]{Lemma}

\newtheorem{proposition}[theorem]{Proposition}

\newtheorem{example}[theorem]{Example}

\newtheorem{assumption}[theorem]{Assumption}

\theoremstyle{definition}

\theoremstyle{remark}
\newtheorem{remark}[theorem]{Remark}

\newtheorem*{acknowledgement}{Acknowledgements}

  1

\newcommand{\GG}{\mathbb{G}}
	\newcommand{\kL}{k_{L}}
\newcommand{\HIW}{H_{\mathrm{Iw}}}

\newcommand{\piL}{\pi_{L}}

\newcommand{\Fin}{\mathfrak{F}_{\infty}}
\newcommand{\Lin}{L_{\infty}}
\newcommand{\GaL}{\Gamma_L}
\newcommand{\hten}{\widehat{\otimes}}
\newcommand{\OF}{o_F}
\newcommand{\MLcris}{\mathcal{D}_{\mathrm{cris},L}}

\title{On the interpolation property of Kupferer's Lubin-Tate regulator maps}

\author{Takamichi Sano and Otmar Venjakob}

\address{Osaka Metropolitan University,
Department of Mathematics,
3-3-138 Sugimoto\\Sumiyoshi-ku\\Osaka\\558-8585,
Japan}
\email{tsano@omu.ac.jp, tkmc310@gmail.com}

\address{Universit\"{a}t Heidelberg,  Mathematisches Institut,  Im Neuenheimer Feld 288,  69120
Heidelberg,  Germany,}
\urladdr{http://www.mathi.uni-heidelberg.de/~venjakob/}
\email{venjakob@mathi.uni-heidelberg.de}
\makeatletter
 
  \@addtoreset{equation}{subsection}
\makeatother

\begin{document}

\begin{abstract}
Generalizing both \cite{LZ} and \cite{SV24,SanoV} we  prove the  interpolation property at Artin characters of Kupferer's regulator map \cite{Ku} over a fixed local field $L$ over $\mathbb{Q}_p$ for $p$ any prime. Since our methods cover the case $p=2,$ we in particular close a gap in the literature even for the (cyclotomic) 2-variable regulator map from \cite{LZ}. Moreover, we study big (dual) Bloch-Kato exponential maps for unramified towers and we use them to improve Colmez' reciprocity law in the Lubin-Tate setting from \cite{SanoV}. Finally, we compare Kupferer's regulator map with the 2-variable regulator map of Loeffler-Zerbes when $L=\mathbb{Q}_p$.
\end{abstract}

\maketitle
\tableofcontents

\section{Introduction}

%Notation:
%\begin{itemize}
%\item $L$: a finite extension of $\QQ_p$,
%\item $\pi_L \in o_L$: a prime element,
%\item $({\rm LT}, +_{\rm LT})$: a Lubin-Tate formal group over $o_L$ for $\pi_L$,
%\item $g_{\rm LT}(Z):= \left( \frac{\partial (X+_{\rm LT} Y)}{\partial X} |_{(X,Y)=(0,Z)}\right)^{-1} \in o_L[[Z]]^\times$,
%\item $\log_{\rm LT}(Z) \in L[[Z]]$: the formal logarithm (which satisfies $\frac{d}{dZ}\log_{\rm LT}(Z) = g_{\rm LT}(Z)$),
%\item $\partial_{\rm inv}:= g_{\rm LT}(Z)^{-1}\frac{d}{dZ}$,
%\item ${\rm LT}_n:= {\rm LT}[\pi_L^n]$: the $\pi_L^n$-torsion points,
%\item $L_n:=L({\rm LT}_n)$,
%\item $L_\infty:=\bigcup_n L_n$,
%\item $T:=\varprojlim_n {\rm LT}_n$: the Tate module of ${\rm LT}$,
%\item $V:=L\otimes_{o_L} T$,
%\item $\cN: o_L((Z)) \to o_L((Z))$: Coleman's norm operator,
%\item $\eta =(\eta_n)_n \in \varprojlim_n {\rm LT}_n = T$: an $o_L$-basis,
%\item $g_u=g_{u,\eta} \in (o_L((Z))^\times)^{\cN=1}$: the Coleman power series for $u \in \varprojlim_n L_n^\times$,
%\item $H^i_{\rm Iw}(L_\infty/L,-):=\varprojlim_n H^i(L_n,-)$: the Iwasawa cohomology.
%\end{itemize}

\subsection{Background}

In \cite{LZ} Loeffler and Zerbes constructed a 2-variable regulator map for crystalline Galois representations building on a 1-variable regulator map attached to the cyclotomic extension of the base field and they showed its interpolation property with regard to Dirichlet characters of $\mathrm{Gal}(K_\infty/\Qp)$ in Theorem 4.15 of (loc.\ cit.), where $K_\infty$ is the composite of $\Qp(\mu_{p^\infty})$ and any unramified $p$-adic Lie extension of $\Qp$.

Let $L$ be a finite extension of $\Qp$, $o_L$ its ring of integers and $\pi_L$ a prime of $o_L$ with residue field $\kL.$
We consider a Lubin-Tate formal group ${\rm LT} = {\rm LT}_{\pi_L}$ over $o_L$ corresponding to the prime element $\pi_L$.
In \cite{SV24} a Lubin-Tate regulator map has been constructed for a certain class of representations with regard to the corresponding Lubin-Tate extension $L_\infty$. Its interpolation property at Artin characters has been established in Corollary 2.6.5 of \cite{SanoV}; for the convenience of the reader we recall it as Theorem \ref{thm:adjointformula} in \S  \ref{sec:regSV}.

Now let $F/L$ be a finite unramified extension,
	$\Fin$ an unramified $p$-adic Lie extension of $F$ and set $\Upsilon:=\Gal(\Fin/F)$.
In \cite{Ku} Kupferer constructed  a regulator map for extensions of the type  $o_L^\times \times\zp\times \Delta$ or $o_L^\times \times \Delta $ over $L,$ where $\Delta$ denotes a finite cyclic group. To state it we denote the $\Qp$-analytic distributions
	of $\Upsilon$ with values in $K$ by $D_{\Qp}(\Upsilon,K)$ and the $L$-analytic
	distributions of $\GaL:=\Gal(L_\infty/L)$ with vales in $K$ by $D_{L}(\GaL,K)$, where $K$ is any complete subfield of $\Cp$ containing the period $\Omega$ of $\LT$ (see \S \ref{sec:notation}) and $\mathfrak{F}_\infty$.

Consider the composite $X_\infty:=\mathfrak{F}_\infty L_\infty$ which is a $p$-adic Lie extension of $F$ with Galois group $G:=\mathrm{Gal}(X_\infty/F)\cong \Gamma_L\times \Upsilon$. We set $\Lambda_{o_L}(G):=o_L[[G]]$. We define the continuous character
\begin{equation*}\label{def tau}
\tau := \chi_{\cyc}\cdot\chi_{\LT}^{-1},
\end{equation*}
where $\chi_{\cyc}$ is the cyclotomic character while $\chi_\LT$ is attached to the Lubin-Tate formal group $\LT$.

Let $T$ be a free $o_L$-modules of finite rank with a continuous linear $G_F$-action such that
% \in \Rep_{o_L,f}^{\cris,\mathrm{an}}(G_F)$ and
$V(\tau^{-1}):=L\otimes_{o_L}T(\tau^{-1})$ is crystalline and analytic with nonnegative Hodge-Tate weights.
If $\mathfrak{F}_\infty/F$ is a finite extension, we assume that $T(\tau^{-1})$ has no quotient isomorphic to the trivial representation. Let $H^1_{\rm Iw}(X_\infty/F,T)$ denote the Iwasawa cohomology for $T$ with respect to the extension $X_\infty/F$. We set $H^1_{\rm Iw}(X_\infty/F,V):=L\otimes_{o_L}H^1_{\rm Iw}(X_\infty/F,T)$.
%\footnote{By the same reasoning as in \cite[Thm.\ 4.4, Prop.\ 4.5]{LZ} the last condition is not required, if $\mathfrak{F}_\infty/F$ is infinite.}.

\begin{theorem}[{Kupferer \cite{Ku}}]\label{thm:Kupferer}
There is a regulator map\footnote{Kupferer takes coefficients in $\Cp,$ but a closer inspection shows immediately that we can take smaller coefficients   here as in \cite{SV24}, see \S \ref{sec:regKu}.}
	\[ \begin{xy} \xymatrix{
	 \mathbf{L}^G_{V}\colon \HIW^1(X_\infty/F,T) \ar[r] & D_L(\GaL,K) \hten_{\widehat{\mathfrak{F}_\infty}} D_{\Qp}(\Upsilon,\widehat{\mathfrak{F}_\infty})
	 \otimes_F  D_{\cris,F}(V(\tau^{-1})),
	 } \end{xy} \]
which is a homomorphism of $\Lambda_{o_L}(G)$-modules.
	This map interpolates the corresponding regulator maps for all finite intermediate fields
	of $\Fin/F$.
	\end{theorem}	

\subsection{Main results}

Now assume $F=L.$ The aim of this paper is
% to extend this definition slightly in order to cover also extensions of the type $o_L^\times \times\zp\times \textit{finite unramified} $ and
to prove its interpolation property   at characters $\chi: G_L\to \mathbb{C}_p^\times.$
 The following result is an adaption of \cite[Thm.\ 4.13]{LZ} and \cite[Cor.\ 6.6.10]{Ku}:  for the convenience of the reader we recall the strategy of its proof in \S \ref{sec:proof}.

We set
$$D_{\Qp}(\Upsilon,\widehat{\mathfrak{F}_\infty})^\circ:=\{\mu\in D_{\Qp}(\Upsilon,\widehat{\mathfrak{F}_\infty})\mid \mu^u=\delta_u\mu \text{ for all $u\in\Upsilon$} \},$$
where $\Upsilon$ acts on the coefficients of $D_{\Qp}(\Upsilon,\widehat{\mathfrak{F}_\infty})$ and $\delta_u$ denotes the element $u$ considered as delta-distribution within $D_{\Qp}(\Upsilon,\widehat{\mathfrak{F}_\infty}).$
Let $E/L$ be a finite extension. For any continuous homomorphism  $\varpi\colon \Upsilon \to o_E^{\times}$ with basis $\eta_{\varpi}$ for $E(\varpi)$,
we point out that there is a natural isomorphism (using that the unramified representation $\varpi$ has its period already in $\widehat{\mathfrak{F}_\infty}$ and forgetting  the basis $\eta_\varpi$)
 \begin{equation}\label{f:Dcrisunramifiedtwist}
\widehat{\mathfrak{F}_\infty}\otimes_L D_{\cris,L}(V(\varpi))\cong E\otimes_{\Qp} \widehat{\mathfrak{F}_\infty}\otimes_L D_{\cris,L}(V)
 \end{equation} as in \cite[before Prop.\ 4.12]{LZ}, which we take as an identification henceforth.
% , in particular in the statement of Theorem \ref{thm:omega-descent} above and in the interpolation formulae below.
Let
$$ D_{\QQ_p}(\Upsilon, \widehat{\mathfrak{F}_\infty}) \to E\otimes_{\QQ_p}\widehat{\mathfrak{F}_\infty}$$
be the map induced by $\varpi^{-1}$. This map together with \eqref{f:Dcrisunramifiedtwist} induces a map
$${\rm pr}_{\rm cris}^{\varpi}: D_L(\GaL,K) \hten_{\widehat{\mathfrak{F}_\infty}} D_{\Qp}(\Upsilon,\widehat{\mathfrak{F}_\infty})^\circ\otimes_L  D_{\mathrm{cris},L}(V(\tau^{-1})) \to D_L(\GaL,K)  \otimes_L D_{\mathrm{cris},L}(V(\varpi\tau^{-1})).$$

\begin{theorem}\label{thm:omega-descent}
For any continuous homomorphism  $\varpi\colon \Upsilon \to o_E^{\times}$ with basis $\eta_{\varpi}$ for $E(\varpi)$, the following diagram commutes:
\[\xymatrix{
  H^1_{\mathrm{Iw}}(X_\infty/L,V) \ar[d]^{\mathrm{pr}_{\mathrm{Iw}}^\varpi} \ar[r]^-{ \mathbf{L}_{V}^{G}} & D_L(\GaL,K) \hten_{\widehat{\mathfrak{F}_\infty}} D_{\Qp}(\Upsilon,\widehat{\mathfrak{F}_\infty})^\circ\otimes_L  D_{\mathrm{cris},L}(V(\tau^{-1})) \ar[d]^{\mathrm{pr}_{\mathrm{cris}}^\varpi} \\
  H^1_{\mathrm{Iw}}({L}_\infty /L,V(\varpi)) \ar[r]^-{\mathbf{L}_{V(\varpi)}^{\Gamma_L}} & D_L(\GaL,K)
  %\otimes_L \widehat{\mathfrak{F}_\infty}
  \otimes_L D_{\mathrm{cris},L}(V(\varpi\tau^{-1})),   }\]
where ${\rm pr}_{\rm Iw}^{\varpi}$ is defined to be the cup product with $\eta_{\varpi}$ followed by the corestriction map with respect to $X_\infty/L_\infty$.
%where the vertical maps arise by twisting followed by corestriction and evaluation at the trivial character $\mathbf{1}\colon\Upsilon\to o_L^\times,$ respectively.
\end{theorem}

For a character $\varrho:G\to o_E^{\times}$ with basis $\eta_\varrho$ for $E(\varrho),$ let $y_{\varrho^{-1},n}$ (and just $y_{\varrho^{-1}}$ for $n=0$) denote the image of $y \in  H^1_{\Iw}(X_\infty/L,T)$ under the map
 \begin{equation}\label{f:twistprojection}
 \mathrm{pr}_{n,\varrho^{-1}}: H^1_{\Iw}(X_\infty/L,T)\xrightarrow{ -\cup \eta_{\varrho^{-1}}
 }H^1_{\Iw}(X_\infty/L,T(\varrho^{-1}))\xrightarrow{\mathrm{cor}}H^1 (
 L_n,T(\varrho^{-1}))\to H^1 ( L_n,V(\varrho^{-1})).
 \end{equation}

From Theorem \ref{thm:omega-descent}, we deduce the following interpolation formula.

%\todo{Action of Frobenius on $\widehat{\mathfrak{F}_\infty}$?}
\begin{theorem}\label{thm:interpolationRegKu} Let $\varpi\colon \Upsilon \to o_E^{\times}$ be a continuous
	homomorphism (also considered as continuous representation of $G$) and assume that $V^*(1)$ is $L$-analytic with $\mathrm{Fil}^{-1}D_{\cris,L}(V^*(1)(\varpi))=D_{\cris,L}(V^*(1)(\varpi))$  and such that $D_{\cris,L}(V^*(1)(\varpi))^{\varphi_L=\pi_L^{-1}}=D_{\cris,L}(V^*(1)(\varpi))^{\varphi_L=1}=0$.
 Let $j\in\mathbb{Z}$   and let $\rho$ be a finite order character of $\Gamma_L$ (whence also of $G$) of conductor $n,$
  i.e., $\rho$ factorizes over $G_n:= \Gal(L_n/L),$ but not over $G_{n-1},$  Then, setting $\varrho= \rho\chi_{\LT}^j\varpi,$ for any $y\in H^1_{\rm Iw}(X_\infty/L,T)$ we have the following:
\begin{enumerate}
\item If $j\geq 0,$ we have  \footnotesize
\begin{align*}
   \Omega^{j} \mathbf{L}_V^{G}(y)(\rho\chi_{\LT}^j)(\varpi)&=  j!
  \begin{cases}
\displaystyle    (1-\pi_L^{-1-j}\varphi_L^{-1})^{-1}(1-\frac{\pi_L^{j+1}}{q}\varphi_L)\Big(\widetilde{\exp}^*_{L,V(\varpi^{-1}\chi_{\LT}^{-j}),\id}(y_{\varrho^{-1}})\otimes e_{ j}\Big), & \text{if  $n=0$,} \\
  %{[L_n:L]  }
 \displaystyle   \pi_L^{n(1+j)}\tau(\rho)^{-1} \varphi_L^{n}  \sum_{\sigma \in G_n}\rho(\sigma)\Big( \widetilde{\exp}^*_{L_n,V(\varpi^{-1}\chi_{\LT}^{-j}),\id}(\sigma^{-1} (y_{\varpi^{-1}\chi_{\LT}^{-j},n}))  \otimes e_{ j}\Big) , & \text{if $n>0$.}
  \end{cases}
\end{align*}\normalsize
\item If $j\leq -1,$ we have \footnotesize
\begin{align*}
 \Omega^{j} \mathbf{L}_V^{G}(y)(\rho\chi_{\LT}^j)(\varpi)&= \\ 
 \frac{{ (-1)^{j+1}}}{(-1-j)!}&
\begin{cases}
 \displaystyle    (1-\pi_L^{-1-j}\varphi_L^{-1})^{-1}(1-\frac{\pi_L^{j+1}}{q}\varphi_L)\Big(\widetilde{\log}_{L,V(\varpi^{-1}\chi_{\LT}^{-j}),\id}(y_{\varrho^{-1}})\otimes e_{ j}\Big), & \text{if  $n=0$,} \\
 %   {[L_n:L]  }
 \displaystyle     \pi_L^{n(1+j)}\tau(\rho)^{-1} \varphi_L^{n}  \sum_{\sigma \in G_n}\rho(\sigma)\Big(\widetilde{\log}_{L_n,V(\varpi^{-1}\chi_{\LT}^{-j}),\id}(\sigma^{-1}(y_{\varpi^{-1}\chi_{\LT}^{-j},n}))  \otimes e_{ j}\Big),  & \text{if $n>0$.}
  \end{cases}
\end{align*}\normalsize
\end{enumerate} 
Here $\mathbf{L}_{V}^G(y)(\rho \chi_\LT^j)(\varpi) \in \CC_p \otimes_L D_{{\rm cris},L}(V(\varpi^{-1}\tau^{-1}))$ denotes the image of $\mathbf{L}_V^G(y)$ under the map induced by $\rho \chi_\LT^j \varpi$, i.e., the image of
$${\rm pr}_{\rm cris}^{\varpi^{-1}} (\mathbf{L}_V^G(y)) \in D_L(\Gamma_L,K) \otimes_L D_{{\rm cris},L}(V(\varpi^{-1}\tau^{-1}))$$
under the map
$$D_L(\Gamma_L,K) \otimes_L D_{{\rm cris},L}(V(\varpi^{-1}\tau^{-1})) \to \CC_p \otimes_L D_{{\rm cris},L}(V(\varpi^{-1}\tau^{-1}))$$
induced by $\rho \chi_\LT^j$.
In both cases, we assume that $1-\pi_L^{-1-j}\varphi_L^{-1}, 1-\frac{\pi_L^{j+1}}{q}\varphi_L $ are invertible on $D_{\cris,L}(V(\varpi^{-1}\tau^{-1}))$ when $\rho$ is the trivial character, i.e., $n=0.$ In addition,  we have
\[\sum_{\sigma \in G_n}\rho(\sigma)\Big( \widetilde{\exp}^*_{L_n,V(\varpi^{-1}\chi_{\LT}^{-j}),\id}(\sigma^{-1} (y_{\varpi^{-1}\chi_{\LT}^{-j},n}))  \otimes e_{ j}\Big)=\widetilde{\exp}^*_{L,V(\varrho^{-1}),\id}( y_{\varrho^{-1}})  \otimes e_{ j}\]
and
\[\sum_{\sigma \in G_n}\rho(\sigma)\Big(\widetilde{\log}_{L_n,V(\varpi^{-1}\chi_{\LT}^{-j}),\id}(\sigma^{-1}(y_{\varpi^{-1}\chi_{\LT}^{-j},n}))  \otimes e_{ j}\Big)=\widetilde{\log}_{L,V(\varrho^{-1}),\id}(y_{\varrho^{-1}})  \otimes e_{ j}.\]
\end{theorem}

\begin{proof}
%From Theorem \ref{thm:Kupferer} (applied to $\varpi^{-1}$), the relation \eqref{f:evaluationtwist} and Theorem \ref{thm:adjointformula} below (applied to $V(\varpi^{-1})$) we derive immediately the desired interpolation formula:
By Theorem \ref{thm:omega-descent}, we have
$$\mathbf{L}_{V}^G(y)(\rho \chi_\LT^j)(\varpi) = \mathbf{L}_{V(\varpi^{-1})}^{\Gamma_L}({\rm pr}_{\rm Iw}^{\varpi^{-1}}(y))(\rho \chi_\LT^j).$$
Combining this with the interpolation formula \cite[Cor.\ 2.6.5]{SanoV} (see Theorem \ref{thm:adjointformula} below) for $\mathbf{L}_{V(\varpi^{-1})}^{\Gamma_L}$, we obtain the claim.
\end{proof}

\begin{remark}\label{rem:CC}
In the cyclotomic case (i.e., when $L=\QQ_p$, $\pi_L=p$, and ${\rm LT}=\widehat \GG_m$) we reobtain \cite[Thm.\ 4.15]{LZ}
even for $p=2$. Unfortunately, we have more conditions than in (loc.\ cit.), which are due to the fact that the interpolation property for the regulator map in the Lubin-Tate setting is obtained by duality from the interpolation property of the big exponential map. If one wants to remove them, one has to prove \cite[Thm.\ B.5]{LZ} for $p=2$, since the process of extending the result to the unramified layer does not require the condition $p$ odd at all. It seems to us that for the formulae in \cite{B}, on which the appendix in \cite{LZ} is based, does not need this condition either.
\end{remark}

\begin{remark}
In \cite{SaV26}, we apply Theorem \ref{thm:interpolationRegKu} to extend interpolation properties of the Katz $p$-adic $L$-function over an imaginary quadratic field.
\end{remark}

Furthermore, in \S \ref{sec:unramiedtowers} -  inspired by \cite[\S 3.2]{vener} and \cite{ochiai}  - we construct under some technical condition a big exponential map
\begin{equation*}
   D_{\dR,F}(V)\otimes_{o_F} S_{\mathfrak{F}_\infty/F} \xrightarrow{{\mathbb{E}\mathrm{xp}}_{\mathfrak{F}_\infty/F,V}} \HIW^1(\mathfrak{F}_\infty/F, V)
\end{equation*}
for the unramified tower $\mathfrak{F}_\infty/F$ as above, which interpolates the usual Bloch-Kato exponential maps for all characters $\varpi:\Upsilon\to E^\times,$ for $E$ some finite extension of $F.$ We also construct a big {\em dual} exponential map
\begin{equation*}
    \HIW^1(\mathfrak{F}_\infty/F, V)\xrightarrow{{\mathbb{E}\mathrm{xp}}^*_{\mathfrak{F}_\infty/F,V }}   D_{\dR,F}(V )\otimes_{o_F} S_{\mathfrak{F}_\infty/F}
\end{equation*}
for the same tower, which similarly interpolates the usual {\em dual} Bloch-Kato exponential maps.

We use these results in \S \ref{sec:Colmez} in order to get rid of conditions in the version of a Colmez style reciprocity law in the Lubin-Tate setting from \cite{SanoV}, see Theorem \ref{thm:ColmzezImproved}. For that purpose we establish in Theorem \ref{thm:bigColmez} a new reciprocity law \`{a} la Colmez for the big dual exponential maps for unramified towers, see diagram \eqref{f:interpolationunramified}.

Finally, in \S \ref{sec:comparison}, we compare Kupferer's regulator map with the 2-variable regulator map of Loeffler-Zerbes \cite{LZ} when $F=L=\QQ_p$. We prove in Theorem \ref{thm:comparison} that they essentially coincide. In particular, we see that Kupferer's regulator map is independent of $\pi_L$ if $\mathfrak{F}_\infty L_\infty = \mathfrak{F}_\infty \QQ_p(\mu_{p^\infty})$. This result can be regarded as a generalization of \cite[Prop. I.3.9]{deshalit} (see also \cite[Thm. 6.11]{LZ}).

\subsection{Notation}\label{sec:notation}

In this paper, we basically use the same notation as in \cite{SV15}, \cite{SV24} and \cite{SanoV}.

%Let $L/\QQ_p$ be a finite   extension of degree $d$, $o_L$ the ring of integers of $L$, $\pi_L \in o_L$ a fixed prime element, $k_L := o_L/\pi_L o_L$ the residue field with cardinality $q := \# k_L$,
%% $e$ the absolute ramification index of $L$
%and $G_L:=\Gal(\overline{L}/L).$ We always use the absolute value $|\cdot |_p$ on $\mathbb{C}_p$ which is normalized by $|\pi_L|_p = q^{-1}$.
%We warn the reader, though, that we will repeatedly use the references \cite{BSX}, \cite{FX}, \cite{La}, \cite{Sc1}, \cite{Sc}, \cite{ST}, and \cite{ST2} in which the absolute value is normalized differently from this paper by $|p| = p^{-1}$. Our absolute value is the $d$th power of the one in these references. The transcription of certain formulas to our convention will usually be done silently.

We consider a Lubin-Tate formal group ${\rm LT} = {\rm LT}_{\pi_L}$ over $o_L$ corresponding to the prime element $\pi_L$.
For the definitions of $g_\LT(Z)$, $\log_\LT(Z)$, and $\partial_{\rm inv}$, see \cite[\S 1.3]{SanoV}.

Let $\TLT$ be the Tate module of ${\rm LT}$. We fix a generator $\eta=(\eta_n)_n$ of $\TLT$. The absolute Galois group $G_L:=\Gal(\overline L/L)$ of $L$ acts on $\TLT$ via the Lubin-Tate character
$$\chi_{\LT} : G_L \longrightarrow o_L^\times.$$

We fix a generator $\eta^{\rm cyc}$ of $\ZZ_p(1)$. Let $\TLT':=\TLT^\ast(1)$ denote the Tate module of the $p$-divisible group Cartier dual to ${\rm LT}$. Let $\Omega$ be the period of $\LT$ associated to the basis $\eta':=(\eta \mapsto \eta^{\rm cyc})$ of $\TLT'$.
%
%Then $\TLT$ is
%a free $o_L$-module of rank one, say with generator $\eta=(\eta_n)_n$, and the action of
%$G_L$ on $\TLT$ is given by a continuous character
%$$\chi_{\LT} : G_L \longrightarrow o_L^\times.$$
%Let $\TLT'$ denote the Tate module of the $p$-divisible group Cartier dual to ${\rm LT}$ with period $\Omega$ (depending on the choice of a generator of $\TLT'$, see \cite{ST2}), which again is a free $o_L$-module of rank one.
The $G_L$-action on $\TLT'$ is given by the continuous character
\begin{equation*}\label{def tau}
\tau := \chi_{\cyc}\cdot\chi_{\LT}^{-1},
\end{equation*}
where $\chi_{\cyc}$ is the cyclotomic character.

Let $K\subseteq \mathbb{C}_p$ be any complete subfield containing the period $\Omega$  and the maximal abelian extension $L^{\rm ab}$ of $L$.

For $n \geq 0$ we let $L_n/L$ denote the extension (in $\mathbb{C}_p$) generated by the $\pi_L^n$-torsion points of ${\rm LT}$, and set $G_n:=\Gal(L_n/L)$. Also, we set $L_\infty := \bigcup_n L_n$, $\Gamma_L := \Gal(L_\infty/L)$, and $H_L := \Gal(\overline{L}/L_\infty)$.
Let $\Rep_{o_L,f}(G_L)$ denote the category of free $o_L$-modules of finite rank  equipped with a continuous linear $G_L$-action.
Let $\Rep_{L}(G_L)$ denote the category of finite dimensional $L$-vector spaces equipped with a continuous linear $G_L$-action.
For $V\in \Rep_L(G_L)$ we put
\[D_{\cris,L}(V) := (B_{\cris} \otimes_{L_0} V)^{G_L}\;\; \mbox{ and } \;\; D_{\dR,L}(V) := ( B_{\dR} \otimes_{\qp} V)^{G_{L}},\]
where $L_0$ denotes the maximal unramified extension of $\QQ_p$ inside $L$.
The second definition will be used also in the more general form
\[D_{\dR,L'}(V) := D_{\dR,L'}(V_{|G_{L'}})=( B_{\dR} \otimes_{\qp} V)^{G_{L'}}\]
for $L'$ any algebraic extension of $L$. Moreover, we abbreviate $D^0_{\dR,L'}(V):=\Fil^0(  D_{\dR,L'}(V)).$
Furthermore we write $D_{\dR,L'}^{\id}(V)$    for the direct summand $( B_{\dR} \otimes_L V)^{G_{L'}}$   of $D_{\dR,L'}(V).$

%Recall that $T_\pi=o_L\eta$ denotes the Tate module of the Lubin-Tate formal group ${\rm LT}$.
We set $V_\pi:=L\otimes_{o_L} T_\pi \in\Rep_L(G_L)$.
Let $t_{\LT}, t_{\QQ_p} \in B_{\dR}$ be the periods for $\LT$ and $\widehat \GG_m$ respectively.
%We write
%$t_{\mathbb{Q}_p} := \log_{\widehat{\mathbb{G}}_m}(\omega_{\widehat{\mathbb{G}}_m})$ for the multiplicative formal group $\widehat{\mathbb{G}}_m$.
Then, for $j\geq 1,$
\begin{equation}\label{def dr}
\mathbf{d}_j := t_{\LT}^j t_{\mathbb{Q}_p}^{-1} \otimes (\eta^{\otimes (-j)} \otimes \eta^{\cyc}),
\end{equation}
is an $L$-basis of  $D^0_{\dR,L}(V_\pi^{\otimes (-j)}(1))$.
We also set $e_j:=t_{\LT}^{-j}\otimes \eta^{\otimes j}\in D_{\cris,L}(L(\chi_{\LT}^j)).$

\section{Lubin-Tate regulator maps}

The aim of this section is to prove Theorem \ref{thm:omega-descent}. For the convenience of the reader, we recall the definition of the Lubin-Tate regulator map of \cite{SV24} in \S \ref{sec:regSV}. We also recall the interpolation property proved in \cite{SanoV} (see Theorem \ref{thm:adjointformula}).  In \S \ref{sec:regKu}, we recall the construction of Kupferer's regulator map \cite{Ku}. Finally, in \S \ref{sec:proof}, we give a proof of Theorem \ref{thm:omega-descent}.
%For the convenience of the reader and in order to clear some details we recall the construction of Kupferer's regulator map. We start with the definition of the Lubin-Tate regulator map of \cite{SV24}.

\subsection{Definition of the regulator map by Schneider-Venjakob}\label{sec:regSV}

We recall regulator maps defined in \cite[\S 5.1]{SV24}.  We use the same notation in loc. cit. In particular, we write
\begin{equation}\label{f:Mellintransform}
 \mathfrak{M}: D(\Gamma_L,K)\xrightarrow{\sim} {\mathcal{O}_{K}(\mathbf{B})}^{\psi_L=0}; \ \lambda\mapsto \lambda\cdot \eta(1,Z)
\end{equation}
for the Mellin transform, where we write $D(\Gamma_L, K)$ for the locally analytic distribution algebra on $\Gamma_L$, $\cO_{K}(\mathbf{B})$ for the ring of rigid analytic functions on the open unit disk $\mathbf{B}$ over $K$, and we set
$$\eta(1,Z):=\exp(\Omega {\rm log}_{\rm LT}(Z))\in \mathcal{O}_{K}(\mathbf{B}).$$
The ring of rigid analytic functions on the open unit disk $\mathbf{B}$ over $L$ is simply denoted by $\cO$. For a subfield $L'\subseteq \CC_p$ we often abbreviate $\cO_{L'}(\mathbf{B})$ to $\cO_{L'}$. Also, we write $\Rep_{o_L,f}^{\cris}(G_L)$ and $\Rep_{o_L,f}^{\cris,\mathrm{an}}(G_L)$ for the full subcategories of $\Rep_{o_L,f}(G_L)$, the    category of free $o_L$-modules of finite rank with a continuous linear $G_L$-action, consisting    of those $T$ which are free over $o_L$ and such that the representation $V := L \otimes_{o_L} T$ is crystalline and in addition   analytic, respectively. Here, $V$ is called {  analytic},   if   the filtration on $D_{\dR}(V)_\mathfrak{m}$ is trivial for each maximal ideal $\mathfrak{m}$ of $L \otimes_{\mathbb{Q}_p} L$ which does not correspond to the identity $\id : L \to L$. Moreover, $D_{\LT}(T)$ denotes the Lubin-Tate $(\varphi_L,\Gamma_L)$-module attached to any $T\in \Rep_{o_L,f}^{\cris}(G_L)$ (see \cite[Thm.\ 1.6]{KR}, also recalled in \cite[Thm.\ 4.2]{SV15}), while   $N(T)$ denotes the attached Wach module \`{a} la Kisin-Ren  for $T$ in $\Rep_{o_L,f}^{\cris,\mathrm{an}}(G_L)$, see \cite[Ch.\ 3, especially before Rem.\ 3.1.6]{SV24}.

Recall that we set $\tau := \chi_{\cyc}\cdot\chi_{\LT}^{-1}$.
 Let   $T$ be in $\Rep_{o_L,f}^{\cris}(G_L)$    such that $T(\tau^{-1})$ belongs to  $\Rep_{o_L,f}^{\cris,\mathrm{an}}(G_L)$ with all   Hodge-Tate weights   in $[0,r],$ and such that $V:=L\otimes_{o_L} T$ does not have any quotient isomorphic to $L(\tau).$ Under these assumptions Schneider and Venjakob \cite[\S 5.1]{SV24} define the regulator map
\begin{align}\label{f:regulatorSV}
\mathbf{L}_V^{\Gamma_L}: & H^1_{\Iw}(L_\infty/L,T)\to  D(\Gamma_L,K)\otimes_L D_{\cris,L}(V(\tau^{-1}))
  %\mathcal{L}_V^0: & H^1_{\Iw}(L_\infty/L,T)\to   \mathcal{O}_L(\mathbf{B})^{\psi_L=0}\otimes_L D_{\cris,L}(V(\tau^{-1})), \\
 % \mathcal{L}_V: & H^1_{\Iw}(L_\infty/L,T)\to D(\Gamma_L,\Cp)\otimes_L D_{\cris,L}(V)
\end{align}
as   the composite \small
\begin{align}\label{f:defregulator}
 \notag H^1_{\Iw}(L_\infty{/L},T)&\stackrel{{\rm Exp}^\ast}{\simeq} D_{\LT}(T(\tau^{-1}))^{\psi_L=1}=N(T(\tau^{-1}))^{\psi_{D_{\LT}(T(\tau^{-1}))}=1} \xrightarrow{1-\frac{\pi_L}{q}\varphi_L}\varphi^*_L(N(V(\tau^{-1})))^{\psi_L=0}\\
  &\hookrightarrow \mathcal{O}^{\psi_L=0} \otimes_L D_{\cris,L}(V(\tau^{-1})) \subseteq {\mathcal{O}_{K}(\mathbf{B})}^{\psi_L=0}\otimes_L D_{\cris,L}(V(\tau^{-1}))\\
\notag &\xrightarrow{\mathfrak{M}^{-1}\otimes \id}D(\Gamma_L,K)\otimes_L D_{\cris,L}(V(\tau^{-1})).
%\to D(\Gamma_L,K)\otimes_L D_{\cris,L}(V),
\end{align}\normalsize
% where the last map sends $\mu\otimes d \in D(\Gamma_L,K)\otimes_L D_{\cris,L}(V(\tau^{-1}))$ to $\mu\otimes d \otimes \mathbf{d}_1\in D(\Gamma_L,K)\otimes_L D_{\cris,L}(V(\tau^{-1}))\otimes_L D_{\cris,L}(L(\tau))\simeq D(\Gamma_L,K)\otimes_L D_{\cris,L}(V). $ Note that
% $$D:=D_{\cris,L}(L(\tau))=D^0_{\dR,L}(L(\tau)) = L \mathbf{d}_1$$
%with $\mathbf{d}_1= t_{\LT} t_{\mathbb{Q}_p}^{-1} \otimes (\eta^{\otimes -1} \otimes \eta^{\cyc})$, where $L(\chi_{\LT})=L\eta$ and $L(\chi_{\cyc})=L\eta^{\cyc}.$

%Alternatively, in order to stress that the regulator is essentially the map $1- \varphi_L, $ one can rewrite this as \small
% \begin{equation*}
%  \begin{aligned}
%  	H^1_{\Iw}&(L_\infty/L,T) \stackrel{{\rm Exp}^\ast}{\simeq} D_{\LT}(V(\tau^{-1}))^{\psi_L=1}=N(T(\tau^{-1}))^{\psi_{D_{\LT}(T(\tau^{-1}))}=1} \\
%  &\hookrightarrow \Big(N(V(\tau^{-1}))\otimes_L D\Big) ^{\Psi =1}\xrightarrow{1- \varphi_L}\varphi^*_L(N(V(\tau^{-1})))^{\Psi_L=0}\otimes_L D \\
%  &\hookrightarrow \mathcal{O}^{\Psi_L=0} \otimes_L D_{\cris,L}(V(\tau^{-1}))\otimes_L D \subseteq {\mathcal{O}_{K}(\mathbf{B})}^{\Psi_L=0}\otimes_L D_{\cris,L}(V)\\
% &\xrightarrow{\mathfrak{M}^{-1}\otimes \id}D(\Gamma_L,K)\otimes_L D_{\cris,L}(V)
%  \end{aligned}
%\end{equation*}\normalsize
% where the $ \hookrightarrow $ in the second line sends $n$ to $n\otimes \mathbf{d}_1$ and the $\varphi_L$ as well as $\Psi$ now acts diagonally.

 As mentioned after \cite[Prop.\ 5.1.1]{SV24} one can extend the definition of the regulator map to the case of all  $T$   in $\Rep_{o_L,f}^{\cris}(G_L)$    such that $T(\tau^{-1})$ belongs to  $\Rep_{o_L,f}^{\cris,\mathrm{an}}(G_L)$, if one is willing to accept denominators in the target, e.g. by replacing $D(\Gamma_L,K)$   by its  total ring of quotients.

For a finite extension $L'/L$, we define
$$\widetilde{\exp}^*_{L',V,\id}: H^1(L',V)\to D_{{\rm cris},L'}(V(\tau^{-1}))$$
by
$$\widetilde{\exp}_{L',V,\id}^*:= \mathbb{T}_{\tau^{-1}}\circ \exp_{L',V}^*,$$
where $\exp^\ast_{L',V}: H^1(L',V)\to D_{{\rm cris},L'}(V)$ is the usual dual exponential map and
\begin{equation}\label{f:dualexpcomp}
\mathbb{T}_{\tau^{-1}}: D_{\cris,L'}(V) \to D_{\cris,L'}(V(\tau^{-1}))
\end{equation}
is the map sending $b\otimes v$ to $b\frac{t_{\mathbb{Q}_p}}{t_{\LT}}\otimes v\otimes \eta\otimes\eta_{\cyc}^{\otimes (-1)}$. Similarly,
$$\widetilde{\log}_{L',V,\id}: H^1(L',V)\to D_{{\rm cris},L'}(V(\tau^{-1}))$$
is defined.

 The following result is  \cite[Cor.\ 2.6.5]{SanoV}.

\begin{theorem}\label{thm:adjointformula}
Assume that $V^*(1)$ is $L$-analytic with $\mathrm{Fil}^{-1}D_{\cris,L}(V^*(1))=D_{\cris,L}(V^*(1))$  and such that $D_{\cris,L}(V^*(1))^{\varphi_L=\pi_L^{-1}}=D_{\cris,L}(V^*(1))^{\varphi_L=1}=0$.
 Let $j\in\mathbb{Z}$   and let $\rho$ be a finite order character of $\Gamma_L$ of conductor $n,$ i.e., $\rho$ factorizes over $G_n,$ but not over $G_{n-1}.$  Then for any $y\in H^1_{\rm Iw}(L_\infty/L,T)$ we have the following:
\begin{enumerate}
\item If $j\geq 0,$ we have \footnotesize
\begin{align*}
   \Omega^{j} \mathbf{L}_V^{\Gamma_L}(y)(\rho\chi_{\LT}^j)&=  j!
\begin{cases}
\displaystyle    (1-\pi_L^{-1-j}\varphi_L^{-1})^{-1}(1-\frac{\pi_L^{j+1}}{q}\varphi_L)\Big(\widetilde{\exp}^*_{L,V(\chi_{\LT}^{-j}),\id}(y_{\chi_{\LT}^{-j}})\otimes e_{ j}\Big), & \text{if  $n=0$,} \\
  %{[L_n:L]  }
 \displaystyle   \pi_L^{n(1+j)}\tau(\rho)^{-1} \varphi_L^{n}  \sum_{\sigma \in G_n}\rho(\sigma)\Big( \widetilde{\exp}^*_{L_n,V(\chi_{\LT}^{-j}),\id}(\sigma^{-1} (y_{\chi_{\LT}^{-j},n}))  \otimes e_{ j}\Big) , & \text{if $n>0$.}
  \end{cases}
\end{align*}\normalsize
\item If $j\leq -1,$ we have\footnotesize
\begin{align*}
 \Omega^{j} \mathbf{L}_V^{\Gamma_L}(y)(\rho\chi_{\LT}^j)&= \frac{{ (-1)^{j+1}}}{(-1-j)!}
\begin{cases}
 \displaystyle    (1-\pi_L^{-1-j}\varphi_L^{-1})^{-1}(1-\frac{\pi_L^{j+1}}{q}\varphi_L)\Big(\widetilde{\log}_{L,V(\chi_{\LT}^{-j}),\id}(y_{\chi_{\LT}^{-j}})\otimes e_{ j}\Big), & \text{if  $n=0$,} \\
 %   {[L_n:L]  }
 \displaystyle     \pi_L^{n(1+j)}\tau(\rho)^{-1} \varphi_L^{n}  \sum_{\sigma \in G_n}\rho(\sigma)\Big(\widetilde{\log}_{L_n,V(\chi_{\LT}^{-j}),\id}(\sigma^{-1}(y_{\chi_{\LT}^{-j},n}))  \otimes e_{ j}\Big),  & \text{if $n>0$.}
\end{cases}
\end{align*}
\end{enumerate}\normalsize
In both cases, we assume that $1-\pi_L^{-1-j}\varphi_L^{-1}, 1-\frac{\pi_L^{j+1}}{q}\varphi_L $ are invertible on $D_{\cris,L}(V(\tau^{-1}))$ when $\rho$ is the trivial character, i.e., $n=0.$ In addition,  we have
\[\sum_{\sigma \in G_n}\rho(\sigma)\Big( \widetilde{\exp}^*_{L_n,V(\chi_{\LT}^{-j}),\id}(\sigma^{-1} (y_{\chi_{\LT}^{-j},n}))  \otimes e_{ j}\Big)=\widetilde{\exp}^*_{L,V(\chi_{\LT}^{-j}\rho^{-1}),\id}( y_{\chi_{\LT}^{-j}\rho^{-1}})  \otimes e_{ j}\]
and
\[\sum_{\sigma \in G_n}\rho(\sigma)\Big(\widetilde{\log}_{L_n,V(\chi_{\LT}^{-j}),\id}(\sigma^{-1}(y_{\chi_{\LT}^{-j},n}))  \otimes e_{ j}\Big)=\widetilde{\log}_{L,V(\chi_{\LT}^{-j}\rho^{-1}),\id}(y_{\chi_{\LT}^{-j}\rho^{-1}})  \otimes e_{ j}.\]
\end{theorem}

\subsection{Definition of Kupferer's regulator map}\label{sec:regKu}

Let $F$ be a finite unramified extension of $L$ and let $\mathfrak{F}_\infty/F$ \footnote{Kupferer assumes that $\Upsilon\cong\zp$, but one easily checks  that the whole construction works in this greater generally, similar as in \cite{LZ}. }  be an unramified $p$-adic Lie extension with Galois group $\Upsilon,$ thus $\Upsilon$ is either a finite cyclic group, or the product of such group with $\zp.$  We assemble the whole situation in the following diagram:%\todo{ ergaenze $L^{ab}$ }
	 \begin{equation}\label{f:unframifieddiagram}
	     \begin{xy} \xymatrix@R-1pc{
	 			& FL^{ab} & \\
				&	X_\infty=\mathfrak{F}_\infty\Lin	\ar@{-}[u]&		\\
&&\\
			\mathfrak{F}_\infty \ar@{-}[uur]_-{\GaL} &		&		\\	
			  &	F\Lin \ar@{-}[uuu]_\Upsilon	&		\\
			F'	\ar@{-}[uu]&		& &	\phantom{F}\Lin \ar@{-}[uuuull] \ar@{-}[ull]	\\
			F \ar@/^0.7cm/@{-}[uuu]^-{\Upsilon}  	\ar@{-}[u]\ar@{-}[uuuuur]^{G}\ar@{-}[uur]_-{\GaL} &	&	\\
				&		& L \ar@{-}[uur]_-{\GaL} \ar@{-}[ull]
	 } \end{xy} \end{equation}
%\ar@/_0.7cm/@{-}[uuuur]_-{G}
%\ar@/_0.7cm/@{-}[uuuu]_-{\GaL}
%\ar@{-}[uur]\ar@/^0.4cm/@{-}[ur]^-{\GaL}
In this setting the Yager module $S_{\mathfrak{F}_\infty/F}$ is defined as follows: For any finite intermediate extension $F\subseteq F'\subseteq \mathfrak{F}_\infty$, we define
$$S_{F'/F}:=\{y \in o_{F'}[\mathrm{Gal}(F'/F)] \mid y^g=[g]y  \text{ for all $g\in \mathrm{Gal}(F'/F)$} \},$$
where $y^g$ denotes the action on the coefficients $o_F$ while $[g]$ denotes the element $g$ considered as element of the group ring. %As in \cite[Prop.\ 3.2, (4)]{LZ} or \cite[Prop.\ 6.4.18]{Ku}\footnote{Kupferer believes that he has to
%	assume $p \neq 2$ for the existence of an integral normal basis generator for finite and unramified Galois
%	extensions over $L$ (i.e. an integral element whose powers are a
%	normal basis). Following Loeffler and Zerbes he refers to  \cite[Thm.\ 4.12]{P} to ensure the existence. But the proof there just regarding the existence works for arbitrary $p$, only the proof of the additional self-duality property requires $p$ to be odd as it is based on \cite[Prop.\ 4.9]{P}. }
One obtains an isomorphism of free, rank one $\Lambda_{o_F}(\Upsilon)$-modules
\begin{equation}\label{f:Yager}
  \varprojlim_{F'} o_{F'} \cong S_{\mathfrak{F}_\infty/F}:=\varprojlim_{F'}S_{F'/F},
\end{equation}
 see Lemma \ref{lem:Ups-cohomology} below.
As in  \cite[Prop.\ 3.6]{LZ} we shall also prove  that
\begin{equation}\label{f:Sinftyproperty}
 S_{\mathfrak{F}_\infty/F}=\{\lambda\in\Lambda_{\widehat{o_{\mathfrak{F}_\infty}}}(\Upsilon) \mid \lambda^u=[u]\lambda\}
\end{equation}
for any topological generator $u$ of $\Upsilon,$ where the first action is just on the coefficients and $[u]$ denotes the element $u$ considered as an element of the Iwasawa algebra. Indeed, define $\mathbb{T}_\Upsilon:=\Lambda_{o_L}(\Upsilon)$ as $\Lambda_{o_L}(\Upsilon)$-module and endow with $G_F$-action via the character $\kappa_\Upsilon:G_F\twoheadrightarrow \Upsilon \to \Lambda_{o_L}(\Upsilon)$ induced by sending $u\in\Upsilon$ to its class $[u^{-1}]$ in the Iwasawa algebra. Then, we endow the completed tensor product $\widehat{o_{\mathfrak{F}_\infty}} \hat{\otimes}_{o_L}\mathbb{T}_\Upsilon$ with the diagonal action. Analogous actions we introduce on the quotients $\widehat{o_{\mathfrak{F}_\infty}} \hat{\otimes}_{o_L}o_{L}[\mathrm{Gal}(F'/F)]=\widehat{o_{\mathfrak{F}_\infty}}  \otimes_{o_L}o_{L}[\mathrm{Gal}(F'/F)].$

\begin{lemma} \label{lem:Ups-cohomology}
Assume that the residue field $\kappa_{\mathfrak{F}_\infty}$ is $p$-closed, i.e., that $\Upsilon$ is infinite.
\begin{enumerate}
  \item  Let $\kappa_F\subseteq \kappa'\subseteq \kappa_{\mathfrak{F}_\infty}$ be a finite intermediate extension, $\varphi_{q'}$ its corresponding arithmetic Frobenius automorphism and also the Frobenius on $F':=W(\kappa')_L[\frac{1}{p}]=W(\kappa')\otimes_{W(\kappa_L)}L$. Then  we have an exact sequence
  \begin{equation}\label{f:Frob}
     \xymatrix@C=0.5cm{
      0 \ar[r] & o_{F'}  \ar[rr]^{} && \widehat{o_{\mathfrak{F}_\infty}} \ar[rr]^{1-\varphi_{q'}} && \widehat{o_{\mathfrak{F}_\infty}} \ar[r] & 0 }
  \end{equation}
  \item  We have canonical isomorphism of $\Lambda_{o_L}(\Upsilon)$-modules
  \begin{align*}
  H^0(\Upsilon,\widehat{o_{\mathfrak{F}_\infty}} \hat{\otimes}_{o_L}o_{L}[\mathrm{Gal}(F'/F)]) & \cong S_{{F}'/F},  \\
     H^0(\Upsilon,\widehat{o_{\mathfrak{F}_\infty}} \hat{\otimes}_{o_L}\mathbb{T}_\Upsilon) & \cong S_{\mathfrak{F}_\infty/F} \cong\Lambda_{o_F}(\Upsilon),  \\
      H^1(\Upsilon,\widehat{o_{\mathfrak{F}_\infty}} \hat{\otimes}_{o_L}\mathbb{T}_\Upsilon)& =0,\\
      \widehat{o_{\mathfrak{F}_\infty}} \hat{\otimes}_{o_F}S_{\mathfrak{F}_\infty/F} &=\widehat{o_{\mathfrak{F}_\infty}} \hat{\otimes}_{o_L}\mathbb{T}_\Upsilon.
  \end{align*}

\end{enumerate}
\end{lemma}
\begin{proof}
 Note that $\widehat{o_{\mathfrak{F}_\infty}}=W(\kappa_{\mathfrak{F}_\infty})_L$ and that we have due to the $p$-closedness an Artin-Schreier type exact sequence of residue fields
 \[\xymatrix@C=0.5cm{
   0 \ar[r] & \kappa' \ar[rr]^{ } &&  \kappa_{\mathfrak{F}_\infty}\ar[rr]^{1-\varphi_{q'}} && \kappa_{\mathfrak{F}_\infty} \ar[r] & 0 }\]
 which lifts by devisage to (ramified) Witt vectors of finite length and then to the claimed statement of (i) by taking projective limits. Right exactness follows as the $R^1\varprojlim$ of an inverse limit of finite modules vanishes.

 In order to show (ii)  we note that for any open subgroup $U\subseteq\Upsilon'$ with $\Upsilon'$ - the Galois group of $\mathfrak{F}_\infty/F'$  - being the maximal pro-$p$-subgroup of $\Upsilon$, i.e., $\Upsilon'\cong \overline{<\varphi_{q'}>}\cong\mathbb{Z}_p,$ we have an exact sequence
\begin{equation*}
  \xymatrix@C=0.5cm{
    0 \ar[r] & o_{\mathfrak{F}_\infty^U}[\Upsilon/\Upsilon']  \ar[rr]^-{\iota} && \widehat{o_{\mathfrak{F}_\infty}}\hat{\otimes}_{o_L}\Lambda_{o_L}(\Upsilon/U) \ar[rrr]^{\varphi_{q'}\otimes\id-\id\otimes [\varphi_{q'}]} && & \widehat{o_{\mathfrak{F}_\infty}}\hat{\otimes}_{o_L}\Lambda_{o_L}(\Upsilon/U), }
\end{equation*}
where the map $\iota$ sends $\sum_{\tau\in \Upsilon/\Upsilon'}  o_\tau \tau$ to $\sum_{\tau\in \Upsilon/\Upsilon'} \left(\sum_{\sigma\in\Upsilon'/U} \sigma^{-1}(o_\tau)\otimes\sigma\right) \tau. $ This is easily checked by writing elements of $\widehat{o_{\mathfrak{F}_\infty}}\hat{\otimes}_{o_L}\Lambda_{o_L}(\Upsilon/U) $ uniquely as $\sum_{\tau\in \Upsilon/\Upsilon'} \sum_{i=0}^{d-1} a_{\tau,i}\otimes [\varphi_{q'}^i]\tau$ and evaluating the conditions of being in the kernel
%and for being in the image, respectively.
in conjunction with \eqref{f:Frob} (see the proof of \cite[Prop.\ 2.1]{V-Kato} for a similar argument for the kernel). Taking inverse limits over $U$ leads to an exact sequence
\begin{equation*}
  \xymatrix@C=0.5cm{
    0 \ar[r] & \varprojlim_U o_{\mathfrak{F}_\infty^U}[\Upsilon/\Upsilon']  \ar[rr]^-{\iota} && \widehat{o_{\mathfrak{F}_\infty}}\hat{\otimes}_{o_L}\Lambda_{o_L}(\Upsilon) \ar[rrr]^{\varphi_{q'}\otimes\id-\id\otimes [\varphi_{q'}]} && & \widehat{o_{\mathfrak{F}_\infty}}\hat{\otimes}_{o_L}\Lambda_{o_L}(\Upsilon) . }
\end{equation*}
%because the first (non-trivial) inverse system has surjective transitions maps (induced by traces of unramified extensions of local fields), whence   $R^1\varprojlim_U o_{\mathfrak{F}_\infty^U}[\Upsilon/\Upsilon']  $ vanishes.
For the surjectivity use that $\Lambda_{o_L}(\Upsilon)=\bigoplus_{\Upsilon/\Upsilon'}\Lambda_{o_L}(\Upsilon')$ to reduce to the case $\Upsilon'=\Upsilon.$ Setting $T=[\varphi_{q'}]-1$ we identify $\Lambda_{o_L}(\Upsilon')$ with $o_L[[T]]$. Then any element $x$ in $ \widehat{o_{\mathfrak{F}_\infty}}\hat{\otimes}_{o_L}\Lambda_{o_L}(\Upsilon') $ can be uniquely written as $\sum_{i\geq 0}  a_i T^i$ with $a_i\in \widehat{o_{\mathfrak{F}_\infty}}$. The calculation
\begin{align*}
\left( \varphi_{q'}\otimes\id-\id\otimes [\varphi_{q'}]\right)(x) & =  \sum_{i\geq 0} \varphi_{q'}( a_i) T^i - \sum_{i\geq 0}  a_i (T+1)T^i\\
    & =  \sum_{i\geq 0}\left( \varphi_{q'}( a_i)-a_i -b_i\right) T^i
   \end{align*}
with $b_0=0$ and $b_i=a_{i-1}$ for $i\geq 1$ shows that $\sum_{i\geq 0} c_i T^i$ is in the image, if and only if the following equations have solutions
\begin{align*}
  c_0 & = (\varphi_{q'}-1)(a_0)\\
 c_1+a_0 & = (\varphi_{q'}-1)(a_1)\\
 \ldots\\
  c_i+a_{i-1} & =(\varphi_{q'}-1)(a_i)
\end{align*}
for all $i\geq 0.$ This is always granted by \eqref{f:Frob}.

Since $[\varphi_{q'}]$ is invertible one sees immediately that   the following sequence is exact
\begin{equation*}
  \xymatrix@C=0.5cm{
    0 \ar[r] & \varprojlim_U o_{\mathfrak{F}_\infty^U}[\Upsilon/\Upsilon']  \ar[rr]^-{\iota} && \widehat{o_{\mathfrak{F}_\infty}}\hat{\otimes}_{o_L}\mathbb{T}_\Upsilon \ar[rrr]^{\varphi_{q'}\otimes\varphi_{q'}-\id\otimes \id} && & \widehat{o_{\mathfrak{F}_\infty}}\hat{\otimes}_{o_L}\mathbb{T}_\Upsilon \ar[r] & 0. }
\end{equation*}
Hence, we obtain that $H^0(\Upsilon',\widehat{o_{\mathfrak{F}_\infty}}\hat{\otimes}_{o_L}\mathbb{T}_\Upsilon)\cong \left(\varprojlim_U o_{\mathfrak{F}_\infty^U}\right)[\Upsilon/\Upsilon']$ and $H^1(\Upsilon',\widehat{o_{\mathfrak{F}_\infty}}\hat{\otimes}_{o_L}\mathbb{T}_\Upsilon)=0.$ Taking $\Upsilon/\Upsilon'$-invariants gives
\[H^0(\Upsilon,\widehat{o_{\mathfrak{F}_\infty}}\hat{\otimes}_{o_L}\mathbb{T}_\Upsilon )=\varprojlim_U o_{\mathfrak{F}_\infty^U}\] (by definition this agrees with $S_{\mathfrak{F}_\infty/F}  $ and a similar proof as for \cite[Prop.\ 2.1]{V-Kato} shows that this module is free of rank one over $\Lambda_{o_F}(\Upsilon)$, compare also with \cite[Prop.\ 3.2, (4)]{LZ} or \cite[Prop.\ 6.4.18]{Ku}\footnote{Kupferer believes that he has to
	assume $p \neq 2$ for the existence of an integral normal basis generator for finite and unramified Galois
	extensions over $L$ (i.e. an integral element whose powers are a
	normal basis). Following Loeffler and Zerbes he refers to  \cite[Thm.\ 4.12]{P} to ensure the existence. But the proof there just regarding the existence works for arbitrary $p$, only the proof of the additional self-duality property requires $p$ to be odd as it is based on \cite[Prop.\ 4.9]{P}. } and \cite[Lem.\ 3.3]{ochiai}). Since
\[H^1(\Upsilon/\Upsilon',H^0(\Upsilon',\widehat{o_{\mathfrak{F}_\infty}}\hat{\otimes}_{o_L}\mathbb{T}_\Upsilon))=H^1(\Upsilon/\Upsilon', \left(\varprojlim_U o_{\mathfrak{F}_\infty^U}\right)[\Upsilon/\Upsilon'] )=0,\]
we conclude the vanishing of $H^1(\Upsilon,\widehat{o_{\mathfrak{F}_\infty}} \hat{\otimes}_{o_L}\mathbb{T}_\Upsilon) $ by using a similar Hochschild-Serre type sequence as in \eqref{f:HS}.

For the first isomorphism in (ii) we note that the open subgroup $U=\mathrm{Gal}(\mathfrak{F}_\infty/F')$ acts trivially on the right factor, whence the $U$-invariants belong to $o_{F'}=\widehat{o_{\mathfrak{F}_\infty}}^{U}$ already and the full $\Upsilon$-invariants reproduce precisely the definition of $S_{{F}'/F}$ (Taking the inverse limit therefore provides another proof of the second isomorphism in (ii)). Finally, the last isomorphism is an immediate consequence of \cite[\S A1, LEMMA, p.\ III-33]{Se0} which implies the existence of a basis of $\widehat{o_{\mathfrak{F}_\infty}}  \otimes_{o_L}o_{L}[\mathrm{Gal}(F'/F)]$  as $\widehat{o_{\mathfrak{F}_\infty}} $-module consisting of $ \Upsilon$-invariant elements. This implies the analogous claim for $S_{F'/F}.$ Taking inverse limits gives the desired result:
\begin{align*}
  \widehat{o_{\mathfrak{F}_\infty}} \hat{\otimes}_{o_F}S_{\mathfrak{F}_\infty/F}   & \cong \varprojlim_{F'} (\widehat{o_{\mathfrak{F}_\infty}} {\otimes}_{o_F}S_{{F}'/F} ) \\
    & \cong  \varprojlim_{F'} (\widehat{o_{\mathfrak{F}_\infty}}  \otimes_{o_L}o_{L}[\mathrm{Gal}(F'/F)] )  \\
    & \cong \widehat{o_{\mathfrak{F}_\infty}} \hat{\otimes}_{o_L}\mathbb{T}_\Upsilon.
\end{align*}
\end{proof}

For $T\in \Rep_{o_L,f}^{\cris,\mathrm{an}}(G_F)$ a {\em relative Kisin-Ren- or Wach-module} $\mathcal{N}_{F'/F}(T)\cong \mathcal{N}_{F'/F}(T)\otimes_{o_F} o_{F'} $ (compatible with $\varphi$-actions, where on the right hand side it is given as $\varphi \otimes \sigma_q$ with the arithmetic $q$-Frobenius $\sigma_q$)\footnote{Prop.\ 6.5.1/Def. 6.5.2 in \cite{Ku} extends immediately to this slightly more general setting.}  as well as $\mathcal{N}_{\mathfrak{F}_\infty/F}(T):=\varprojlim_{F'} \mathcal{N}_{F'/F}(T)$ is introduced satisfying the following properties (see \cite[Prop.\ 6.5.10, Lem.\ 6.6.1, Lem.\ 6.6.2, Lem.\ 6.6.4]{Ku} and \cite[Prop.\ 3.13]{LZ}):
\begin{lemma}\label{lem:KR-modules} For the relative Kisin-Ren-Wach modules the following holds:
\begin{enumerate}
  \item  There is an isomorphism of topological modules $\mathcal{N}_{\mathfrak{F}_\infty/F}(T)\cong \mathcal{N}_{F/F}(T)\hten_{o_F} S_{\mathfrak{F}_\infty/F}$ using \eqref{f:Yager}.
  \item  There is an isomorphism of $\Lambda_{o_F}(G)$-modules $H^1_{\mathrm{Iw}}(X_\infty/F,T')\cong \mathcal{N}_{\mathfrak{F}_\infty/F}(T'(\tau^{-1}))^{\psi=1}$ for $T'$ as in Theorem \ref{thm:Kupferer}.
  \item There is a homomorphism\footnote{From the definition of $\varphi^*\mathcal{N}_{F'/F}$  in \cite[6.5.5]{Ku} one sees immediately, that $\varphi^*\mathcal{N}_{F'/F}\cong (\varphi^*\mathcal{N}_{F/F})\otimes_{o_F} o_{F'},$ whence $(\varphi^* \mathcal{N}_{\mathfrak{F}_\infty/F}(T))^{\psi=0}\cong \varprojlim_{F'} (\varphi^* \mathcal{N}_{{F}/F}(T))^{\psi=0} \otimes_{o_F} o_{F'}\cong (\varphi^* \mathcal{N}_{F/F}(T))^{\psi=0}\hten_{o_F} S_{\mathfrak{F}_\infty/F}$ since $\psi$ acts bijectively on $o_{F'}.$}%\todo{Tensorprodukte \"{u}ber den richtigen Ringen?}
      \[(\varphi^* \mathcal{N}_{\mathfrak{F}_\infty/F}(T))^{\psi=0}
  \to (\varphi^* \mathcal{N}_{F/F}(T))^{\psi=0}\hten_{o_F} S_{\mathfrak{F}_\infty/F}
  \to (\mathcal{O}_{K}(\mathbf{B}))^{\psi=0}\hten_{o_F}S_{\mathfrak{F}_\infty/F} \otimes_{o_F} D_{\cris,F}(V),
  \]
 % \item $(\mathcal{O}_{K}(\mathbf{B})\hten_{o_F}S_{\mathfrak{F}_\infty/F})^{\psi=0}\cong (\mathcal{O}_{K}(\mathbf{B}))^{\psi=0}\hten_{o_F}S_{\mathfrak{F}_\infty/F}$,
  \item We have an injective homomorphism   assuming $\mathfrak{F}_\infty\subseteq K$
  \[(\mathcal{O}_{K}(\mathbf{B}))^{\psi=0}\hten_{o_F}S_{\mathfrak{F}_\infty/F}\hookrightarrow D_L(\Gamma_L,K)\hten_{\widehat{{\mathfrak{F}_\infty}}} D_{\Qp}(\Upsilon,\widehat{\mathfrak{F}_\infty})\]
  induced from the natural inclusions $ \widehat{o_{\mathfrak{F}_\infty}} \hat{\otimes}_{o_L}S_{\mathfrak{F}_\infty/F}\subseteq \Lambda_{\widehat{o_{\mathfrak{F}_\infty}}}(\Upsilon)\subseteq D_{\Qp}(\Upsilon,\widehat{\mathfrak{F}_\infty})$ in Lemma \ref{lem:Ups-cohomology} (ii) and from \eqref{f:Mellintransform}.
  \item If $M$ denotes a finite free $o_L$-module with continuous action of $\Upsilon$ via a homomorphism $\Lambda_{o_L}(\Upsilon)\to \mathrm{End}_{o_L}(M),$ then $ \mathcal{N}_{ F/F}(M)=\mathbb{A}_F^+\otimes_{o_F}\left(\widehat{o_{\mathfrak{F}_\infty}}\otimes_{o_L}M\right)^\Upsilon.$
\end{enumerate}
\end{lemma}

Henceforth we assume that $\mathfrak{F}_\infty\subseteq K$. Now let $T$ be as in Theorem \ref{thm:Kupferer}. Then the regulator map
\[\mathbf{L}_V^G\colon H^1_{\mathrm{Iw}}(X_\infty/F,T)\to D_L(\Gamma_L,K)\hten_{o_F} D_{\Qp}(\Upsilon,\widehat{\mathfrak{F}_\infty})\otimes_{F} D_{\cris,F}(V(\tau^{-1}))\] is defined as the composite of the maps from above:
\begin{align*}
  H^1_{\mathrm{Iw}}(X_\infty/F,T)& \xrightarrow{\phantom{m}\cong\phantom{m.}} \mathcal{N}_{\mathfrak{F}_\infty/F}(T(\tau^{-1}))^{\psi=1} \\
    & \xrightarrow{1-\frac{\pi_L}{q}\varphi} (\varphi^* \mathcal{N}_{\mathfrak{F}_\infty/F}(T(\tau^{-1})))^{\psi=0}  \\
  %  &  \xrightarrow{\phantom{1-\frac{\pi_L}{q_L}}}  (\mathcal{O}_{K}(\mathbf{B})\hten_{o_F}S_{\mathfrak{F}_\infty/F})^{\psi=0} \otimes_{F} D_{\cris,F}(V(\tau^{-1}))\\
    &  \xrightarrow{\phantom{1-\frac{\pi_L}{q_L}}}   (\mathcal{O}_{K}(\mathbf{B}))^{\psi=0}\hten_{o_F}S_{\mathfrak{F}_\infty/F} \otimes_{F} D_{\cris,F}(V(\tau^{-1}))\\
    &  \xrightarrow{\phantom{1-\frac{\pi_L}{q_L}}}  D_L(\Gamma_L,K)\hten_{\widehat{{\mathfrak{F}_\infty}}} D_{\Qp}(\Upsilon,\widehat{\mathfrak{F}_\infty})\otimes_{F} D_{\cris,F}(V(\tau^{-1}))
\end{align*}
For the second map we use \cite[Rem.\ 6.5.5]{Ku}  for $F'$ and take inverse limits.

\begin{theorem}\label{thm:descent}
  Let $T$ be as in Theorem \ref{thm:Kupferer}. Then the regulator map $\mathbf{L}_V^G$  is a homomorphism
	of $\Lambda_{o_L}(G)$-modules and has the following two properties, which uniquely determine this
	homomorphism:
	\begin{enumerate}
		\item If $F'/F$ is a finite, unramified extension contained in $\mathfrak{F}_\infty$, we get a commutative
		diagram
		\[ \begin{xy} \xymatrix@C-0.7pc{
			\HIW^1(X_\infty/F,T) \ar[rr]^-{\mathbf{L}_V^G} \ar[d]& &
			D_L(\GaL,K) \hten_{\widehat{{\mathfrak{F}_\infty}}} D_{\Qp}(\Upsilon,\widehat{\mathfrak{F}_\infty})\otimes_{L} \MLcris(V(\tau^{-1}))\ar[d] \\
			\HIW^1(F'\Lin/F,T) \ar[rr]^-{\mathbf{L}_V^{G'}} & &
			D_L(\GaL,K) \hten_{\widehat{{\mathfrak{F}_\infty}}} D_{\Qp}(\Upsilon',\widehat{\mathfrak{F}_\infty}) \otimes_{L} \MLcris(V(\tau^{-1})),
		} \end{xy} \]
		where $G'=\mathrm{Gal}(F'L_\infty/F)$,   $\Upsilon'=\mathrm{Gal}(F'/F)$ and by abuse of notation we have added the injection  
\[D_L(\GaL,K) \hten_{F'} D_{\Qp}(\Upsilon',F') \otimes_{L} \MLcris(V(\tau^{-1}))\hookrightarrow D_L(\GaL,K) \hten_{\widehat{{\mathfrak{F}_\infty}}} D_{\Qp}(\Upsilon',\widehat{\mathfrak{F}_\infty}) \otimes_{L} \MLcris(V(\tau^{-1}))\] 
 after the original $\mathbf{L}_V^{G'}$ in the second line. Note that we can recover $\mathbf{L}_V^{G'} $ from the regulator map $\mathbf{L}_V^{\Gamma_L} $ in \eqref{f:regulatorSV}(or rather its generalization over $F$ instead of $L$):
\[\mathbf{L}_V^{G'} (x)=\sum_{\tau\in \mathrm{Gal}(F'/F)}^{ } [\tau] \mathbf{L}_V^{\Gamma_L} (\tau^{-1}(x)).\]
		\item For any $x \in \HIW^1(X_\infty/F,T)$ and any character
		$\eta\colon \GaL \to \Cp^\times$ the distribution on $\Upsilon$, which is defined by twisting $\mathbf{L}_V^G(x)$ by $\eta$ and pushing forward along the projection to $\Upsilon,$
		is bounded.
	\end{enumerate}
\end{theorem}
\begin{proof}
  \cite[Thm.\ 6.6.7]{Ku} and \cite[Thm.\ 4.7]{LZ}: for each arrow in the definition above one has to check that  they induce commutative descent diagrams.
\end{proof}

From the construction and property \eqref{f:Sinftyproperty} one sees immediately that actually $\mathbf{L}_V^G$ takes values in the subset
 $D_L(\Gamma_L,K)\hten_{\widehat{{\mathfrak{F}_\infty}}} D_{\Qp}(\Upsilon,\widehat{\mathfrak{F}_\infty})^\circ\otimes_{F} D_{\cris,F}(V(\tau^{-1}))$ (recall the notation introduced before Theorem \ref{thm:omega-descent}).
\subsection{Proof of Theorem \ref{thm:omega-descent}}\label{sec:proof}

\begin{lemma} Let $\varpi: G_F \twoheadrightarrow \Upsilon \to o_E^\times$ be an (unramified) character with $E$ some finite extension of $F$ and let $F'/F$ be a finite unramified extension. Then, for $T\in \Rep_{o_L,f}^{\cris,\mathrm{an}}(G_F)$ and $\Rep_{o_L,f}^{\cris,\mathrm{an}}(G_{F'})$ also
\begin{enumerate}
  \item $T'=T(\varpi)$ and
  \item $T''=\mathrm{Ind}^{F'}_F T$.
\end{enumerate}
belong to $\Rep_{o_L,f}^{\cris,\mathrm{an}}(G_F)$ and  $\Rep_{o_L,f}^{\cris,\mathrm{an}}(G_F)$, respectively.  If $T$ is crystalline, respectively nonnegative, then   $T'$ and $T''$ are so, too.
\end{lemma}

\begin{proof}
   (ii) is not needed in this paper and left as an exercise. The reason for (i) and the additional statement is, that a twist by an unramified character does not change the Hodge-Tate weights (for any embedding of $E$ into $\Cp$.)
\end{proof}

For simplicity we assume $F=L$ henceforth. We set $B_{\mathrm{cris},L}:=B_{\mathrm{cris}}\otimes_{L_0}L,$ where $L_0$ denotes the maximal absolute unramified subfield of $L.$ Note that $\widehat{\mathfrak{F}_\infty}\subseteq B_{\mathrm{cris},L}.$
Then, for a crystalline, $L$-analytic representation $W$ of $G_L$ and a continuous character $\varpi\colon \Upsilon \to o_E^{\times}$ we let $b_\varpi$ denote the unique isomorphism
\[\widehat{\mathfrak{F}_\infty}\otimes_{L} \mathcal{D}_{\mathrm{cris},L}(W)\otimes_L E(\varpi) \cong \widehat{\mathfrak{F}_\infty}\otimes_{L} \mathcal{D}_{\mathrm{cris},L}(W(\varpi))\]
such that extending scalars gives a commutative diagram
\[\xymatrix{
  B_{\mathrm{cris},L}\otimes_{L}\mathcal{D}_{\mathrm{cris},L}(W)\otimes_L E(\varpi) \ar[d]_{\mathrm{can}(W)\otimes \id} \ar[r]^-{b_\varpi} & B_{\mathrm{cris},L}\otimes_{L}\mathcal{D}_{\mathrm{cris},L}(W(\varpi)) \ar[d]^{\mathrm{can}(W(\varpi))} \\
  B_{\mathrm{cris},L}\otimes_{L}W\otimes_L E(\varpi) \ar[r]^-{\id} & B_{\mathrm{cris},L}\otimes_{L}W(\varpi)   }\]
where $\mathrm{can}(W)\colon  B_{\mathrm{cris},L}\otimes_{L}\mathcal{D}_{\mathrm{cris},L}(W)\cong B_{\mathrm{cris},L}\otimes_{L}W$ denotes the canonical isomorphism induced from the classical comparison isomorphism $ B_{\mathrm{cris}}\otimes_{L_0}\mathcal{D}^{\Qp}_{\mathrm{cris},L}(W)\cong B_{\mathrm{cris}}\otimes_{\Qp}W $ with $\mathcal{D}^{\Qp}_{\mathrm{cris},L}(W):=(B_{\mathrm{cris}}\otimes_{\Qp} W)^{G_L}.$ The existence follows from \cite[Prop.\ 6.2.24]{Ku}, i.e., from the analogue of \cite[Prop.\ 3.8]{LZ}, implying that $\widehat{\mathfrak{F}_\infty}\otimes_{L} E $ contains a period of $\varpi$.
Together with the canonical twisting map
\[D_{\Qp}(\Upsilon,\widehat{\mathfrak{F}_\infty}) \otimes_{L} E\to D_{\Qp}(\Upsilon,\widehat{\mathfrak{F}_\infty}) \otimes_{L} E \]
which is induced by sending $u\mapsto \varpi(u)^{-1}u$ for $u\in\Upsilon,$ we obtain a map
\begin{align*}
  \mathrm{Tw}_{\varpi^{-1}}:E \otimes_L D_L(\GaL,K) \hten_{\OF} D_{\Qp}(\Upsilon,&\widehat{\mathfrak{F}_\infty}) \otimes_{L} \MLcris(V(\tau^{-1}))\\ & \to  D_L(\GaL,K) \hten_{\OF} D_{\Qp}(\Upsilon,\widehat{\mathfrak{F}_\infty})  \otimes_{L} \MLcris(V(\tau^{-1}\varpi)),
\end{align*}
which has the property regarding evaluation at the trivial character $\mathbf{1}\colon \Upsilon\to o_L^\times$
\begin{equation}\label{f:evaluationtwist}
   \mathrm{Tw}_{\varpi^{-1}}(\mu)(\mathbf{1})=\mu(\varpi^{-1})
\end{equation}
as elements in $D_L(\GaL,K)   \otimes_{L} \widehat{\mathfrak{F}_\infty}\otimes_{L}\MLcris(V(\tau^{-1}\varpi)). $
%where $\mathrm{can}(W)\colon  B_{\mathrm{cris}}\otimes_{L_0}\mathcal{D}_{\mathrm{cris},L}(W)\cong B_{\mathrm{cris}}\otimes_{L_0}W$ denotes the canonical isomorphism induced from the classical comparison isomorphism $ B_{\mathrm{cris}}\otimes_{L_0}\mathcal{D}^{\Qp}_{\mathrm{cris},L}(W)\cong B_{\mathrm{cris}}\otimes_{\Qp}W $ with $\mathcal{D}^{\Qp}_{\mathrm{cris},L}(W):=(B_{\mathrm{cris}}\otimes_{\Qp} W)^{G_L}.$

As mentioned in \cite[Prop.\ 6.6.9]{Ku} the proof of the next result is completely analogous to that of \cite[Prop.\ 4.12]{LZ}:
\begin{proposition}
  \label{prop:omegatwist}
	Let $E/L$ be a finite extension and $\varpi\colon \Upsilon \to o_E^{\times}$ a continuous
	homomorphism. Then the following $G$-equivariant diagram commutes
	\[ \begin{xy} \xymatrix@C-1pc{
		 o_E(\varpi) \otimes_{o_L} \HIW^1(X_\infty/L, T) \ar[r]^-{\id\otimes\mathbf{L}_V^G} \ar[d]^{\cong} &
		E(\varpi) \otimes_L D_L(\GaL,K) \hten_{\widehat{\mathfrak{F}_\infty}} D_{\Qp}(\Upsilon,\widehat{\mathfrak{F}_\infty}) \otimes_{L} \MLcris(V(\tau^{-1})) \ar[d]^{\mathrm{Tw}_{\varpi^{-1}}} \\
		 \HIW^1(X_\infty/L, T(\varpi)) \ar[r]^-{\mathbf{L}^G_{V(\varpi)}} &
		D_L(\GaL,K) \hten_{\widehat{\mathfrak{F}_\infty}} D_{\Qp}(\Upsilon,\widehat{\mathfrak{F}_\infty})  \otimes_{L} \MLcris(V(\tau^{-1}\varpi)).
	} \end{xy} \]
\end{proposition}
Now Theorem \ref{thm:omega-descent} is an immediate consequence of Proposition \ref{prop:omegatwist} and Theorem \ref{thm:descent}.

\section{(Dual) Bloch-Kato exponential maps in unramified towers}\label{sec:unramiedtowers}

In this section, we improve Colmez' reciprocity law in the Lubin-Tate setting proved in \cite[Thm.\ 3.2.1]{SanoV}.

\subsection{Construction of big (dual) Bloch-Kato exponential maps}

The following treatment of deforming the (dual) Bloch-Kato exponential maps along {\it unramified} towers is inspired by \cite[\S 3.2]{vener} and \cite{ochiai} and was gratefully suggested  to us by David Loeffler.

We keep the notation from the beginning of \S \ref{sec:regKu} as illustrated in diagram \eqref{f:unframifieddiagram} but we want to stress that no data of any Lubin-Tate formal group is needed in this subsection. We write 
$$R:=
\Lambda_{o_L}(\Upsilon)\otimes_{o_L}L.$$
Let $T$ be in $\Rep_{o_L}(G_F)$, set $V:=T\otimes_{o_L}L$ and $A:=T\otimes_{o_L} L/o_L.$ The corresponding Kummer duals will be denoted by $T^*(1):=\Hom_{o_L}(T,o_L)(1),$ $V^*(1)=\Hom_{L}(V,o_L)(1)=T^*(1)\otimes_{o_L}L$ and $A^*(1):=T^\vee(1):=\Hom_{o_L}(T,L/o_L)(1)=T^*(1)\otimes_{o_L} L/o_L.$ For the first part of this section we will assume the following:

\begin{assumption}\label{assump}
$H^0(G_{\mathfrak{F}_\infty},T)=0$.
%\begin{equation}\label{f:assumption}
%\mbox{\bf  {Assumption I:     }} \;\;\;
%H^0(G_{\mathfrak{F}_\infty},T)=0.
%\phantom{mmmmmmmmmmmmmmmmmmmmmmmm}
%\end{equation}
\end{assumption}

We recall that the Iwasawa cohomology $\HIW^1(\mathfrak{F}_\infty/F, V)=\HIW^1(\mathfrak{F}_\infty/F, T)\otimes_{o_L} L$ is isomorphic to $\H^1(F,\mathbb{T}_\Upsilon(T))\otimes_{o_L} L=\H^1(F,\mathbb{T}_\Upsilon(V))$, where $\mathbb{T}_\Upsilon(T):=T\otimes_{o_L}\mathbb{T}_\Upsilon$ with diagonal $G_F$-action, similarly for $\mathbb{T}_\Upsilon(V).$

For the definition of the big dual exponential map the following statement will be crucial.

\begin{proposition}\label{prop:Iw-perfectDual}\
 \begin{enumerate}
   \item $ \HIW^1(\mathfrak{F}_\infty/F, V)$ is a projective $R$-module of rank $\dim_{L}(V)[F:\mathbb{Q}_p]$.
   \item The cup product
   \[\langle ,\rangle_{\Tate,R}: \HIW^1(\mathfrak{F}_\infty/F, V) \times \HIW^1(\mathfrak{F}_\infty/F, V^*(1))\xrightarrow{\cup} R\]
   induces an isomorphism of $R$-modules
   \[\HIW^1(\mathfrak{F}_\infty/F, V)\cong\Hom_R(\HIW^1(\mathfrak{F}_\infty/F, V^*(1)),R).\]
 \end{enumerate}
\end{proposition}

\begin{remark}
Note that the statements of the proposition are, in general, not symmetric, i.e., they do not apply necessarily to $V^*(1)$ as long as Assumption \ref{assump} is not satisfied for $T^*(1),$ too.
\end{remark}

\begin{proof}
 The first assertion follows from \cite[Prop.\ 3.2]{venComp} applied to $A^*(1)$ (instead of $A$ there) by inverting $p$ and noting that the module $M$ in (loc.\ cit.) is finite.
 The second item is a direct consequence of \cite[Prop.\ 8.11.4 (i)]{nek} after inverting $p$ upon noting that the long exact sequence there is attached to Jansen's spectral sequence ((8.11.2.4) in loc.\ cit.), which in turn is induced by Tate's local duality. Note that kernel and cokernel of the homomorphism in question vanish by Prop.\ 8.11.5 (i) in (loc.\ cit.) due to our Assumption \ref{assump} and finiteness of the $E^2$-term.
\end{proof}

\begin{remark}
Under the stronger assumption that   the quotient $H^0(G_{\mathfrak{F}_\infty},A)_\mathrm{div}$ of $H^0(G_{\mathfrak{F}_\infty},A)$ by its maximal $p$-divisible submodule vanishes the same references imply the projectivity and duality result in the previous proposition even without inverting $p.$
\end{remark}

Henceforth we assume, in addition, that $V$ is de Rham and consider the Bloch-Kato exponential map
\begin{equation}\label{f:BKFinfty}
   \exp_{\mathfrak{F}_\infty,V}:  D_{\dR,\mathfrak{F}_\infty}(V) \to  H^1(\mathfrak{F}_\infty,V),
\end{equation}
completed   base change of which with respect to $\mathbb{T}_\Upsilon:=\Lambda_{o_L}(\Upsilon)$ gives
\begin{equation}\label{f:basechanged}
    \exp_{\mathfrak{F}_\infty,V}\hat{\otimes}_{o_L}\mathbb{T}_\Upsilon:  D_{\dR,\mathfrak{F}_\infty}(V)\hat{\otimes}_{o_L}\mathbb{T}_\Upsilon \to  H^1(\mathfrak{F}_\infty,V)\hat{\otimes}_{o_L}\mathbb{T}_\Upsilon.
\end{equation}
\begin{lemma}\label{lem:completedbasechange} We have a canonical isomorphism of $\Upsilon$-modules
\begin{equation}\label{f:tensorcohom}
    H^1(\mathfrak{F}_\infty,V)\hat{\otimes}_{o_L}\mathbb{T}_\Upsilon\cong  H^1(\mathfrak{F}_\infty, \mathbb{T}_\Upsilon(V)).
\end{equation}
\end{lemma}
\begin{proof}
For $U$ running through the open subgroups of $\Upsilon$ we obtain
   \begin{align*}
     H^1(\mathfrak{F}_\infty,T)\hat{\otimes}_{o_L}\mathbb{T}_\Upsilon & \cong H^1(\mathfrak{F}_\infty,T)\hat{\otimes}_{o_L}\varprojlim_U(\mathbb{T}_\Upsilon)_U\\
       & \cong \varprojlim_U \left(H^1(\mathfrak{F}_\infty,T)\hat{\otimes}_{o_L}(\mathbb{T}_\Upsilon)_U\right)\\
       & \cong \varprojlim_U \left(H^1(\mathfrak{F}_\infty,T){\otimes}_{o_L}(\mathbb{T}_\Upsilon)_U\right)\\
       & \cong \varprojlim_U H^1(\mathfrak{F}_\infty,T{\otimes}_{o_L}(\mathbb{T}_\Upsilon)_U),
   \end{align*}
where we use for the second isomorphism that completed tensor products commute with projective limits and for the third isomorphism that the completed tensor product coincides with the usual one if one of the factors is finitely generated. Since $(\mathbb{T}_\Upsilon)_U\cong o_L[\Upsilon/U] $ is a finite free $o_L$-module with trivial $G_{\mathfrak{F}_\infty}$-action, the tensor product commutes with Galois cohomology. By an analogous reasoning as in the proof of \cite[Lem.\ 5.8]{SV15} we obtain a short exact sequence
\begin{multline*}
    0 \longrightarrow R^1 \varprojlim_U H^{0}_{}(\mathfrak{F}_\infty, T{\otimes}_{o_L}(\mathbb{T}_\Upsilon)_U) \longrightarrow \\
    H^1_{}(\mathfrak{F}_\infty, T{\otimes}_{o_L}\mathbb{T}_\Upsilon) \longrightarrow \varprojlim_U H^1_{}(\mathfrak{F}_\infty, T{\otimes}_{o_L}(\mathbb{T}_\Upsilon)_U) \longrightarrow 0  .
\end{multline*}
Since $ H^{0}_{}(\mathfrak{F}_\infty, T{\otimes}_{o_L}(\mathbb{T}_\Upsilon)_U)= H^{0}_{}(\mathfrak{F}_\infty, T){\otimes}_{o_L}(\mathbb{T}_\Upsilon)_U=0$ for all $U$ by Assumption \ref{assump}, we thus obtain an isomorphism
\[H^1(\mathfrak{F}_\infty,T)\hat{\otimes}_{o_L}\mathbb{T}_\Upsilon\cong H^1_{}(\mathfrak{F}_\infty, T{\otimes}_{o_L}\mathbb{T}_\Upsilon) ,\]
from which the claim follows by inverting $p.$
\end{proof}
Using the identification
\begin{equation}\label{f:DdR}
   D_{\dR,\mathfrak{F}_\infty}(V)\cong D_{\dR,F}(V)\otimes_{o_F}\widehat{o_{\mathfrak{F}_\infty}}
\end{equation}
and Lemma \ref{lem:completedbasechange}, we can rewrite \eqref{f:basechanged} as
\begin{equation}\label{f:basechangedII}
    \exp_{\mathfrak{F}_\infty,V}\hat{\otimes}_{o_L}\mathbb{T}_\Upsilon:  D_{\dR,F}(V)\otimes_{o_F}\widehat{o_{\mathfrak{F}_\infty}} \hat{\otimes}_{o_L}\mathbb{T}_\Upsilon\to  H^1(\mathfrak{F}_\infty, \mathbb{T}_\Upsilon(V)).
\end{equation}
Finally, we take $\Upsilon$-invariants and obtain
\begin{equation}\label{f:defExp}
  {\mathbb{E}\mathrm{xp}}_{\mathfrak{F}_\infty/F,V}:=( \exp_{\mathfrak{F}_\infty,V}\hat{\otimes}_{o_L}\mathbb{T}_\Upsilon)^{\Upsilon}:D_{\dR,F}(V)\otimes_{o_F}\left(\widehat{o_{\mathfrak{F}_\infty}} \hat{\otimes}_{o_L}\mathbb{T}_\Upsilon \right)^{\Upsilon} \to  \left(H^1(\mathfrak{F}_\infty, \mathbb{T}_\Upsilon(V))\right)^{\Upsilon} .
\end{equation}
Although continuous cohomology does not, in general, form a derived functor and there is, in general, no Hochschild-Serre spectral sequence, one has  by \cite[Thm.\ 10.25]{Th1} a short exact sequence in low degrees
\begin{equation}\label{f:HS}
   \xymatrix@C=0.5cm{
    0 \ar[r] & \H^1(\Upsilon,\H^0(\mathfrak{F}_\infty,\mathbb{T}_\Upsilon(V))) \ar[r]^{ } &  H^1(F, \mathbb{T}_\Upsilon(V))  \ar[r]^{ } & \left(H^1(\mathfrak{F}_\infty, \mathbb{T}_\Upsilon(V))\right)^{\Upsilon} \ar[r] & 0, }
\end{equation}
because $ \H^2(\Upsilon,\H^0(\mathfrak{F}_\infty,\mathbb{T}_\Upsilon(V)))$ vanishes as the $p$-cohomological dimension of $\Upsilon$ equal $1$. Since $\H^0(\mathfrak{F}_\infty,\mathbb{T}_\Upsilon(V))\cong \H^0(\mathfrak{F}_\infty,V)\hat{\otimes}_{o_L}\mathbb{T}_\Upsilon=0$, we may and do identify
$$H^1(F, \mathbb{T}_\Upsilon(V)) \cong \left(H^1(\mathfrak{F}_\infty, \mathbb{T}_\Upsilon(V))\right)^{\Upsilon}.$$
Together with Lemma \ref{lem:Ups-cohomology}(ii) we thus arrive at
\begin{equation}\label{f:defExpfinal}
  {\mathbb{E}\mathrm{xp}}_{\mathfrak{F}_\infty/F,V}:D_{\dR,F}(V)\otimes_{o_F} S_{\mathfrak{F}_\infty/F} \to  H^1(F, \mathbb{T}_\Upsilon(V))\cong \HIW^1(\mathfrak{F}_\infty/F, V).
\end{equation}

Now let $\varpi: G_F \twoheadrightarrow \Upsilon \to o_E^\times$ be an (unramified) character with $E$ some finite extension of $F$. By abuse of notation we shall also denote the induced $o_L$-algebra homomorphism $R\to E$ by $\varpi.$ The next lemma describes the behaviour of the Yager-module under specialization.
\begin{lemma}\label{lem:varpidescent} We have canonical isomorphisms
\begin{enumerate}
  \item $\mathbb{T}_\Upsilon \otimes_{\Lambda_{o_L}(\Upsilon),\varpi}E\cong   E(\varpi^{-1}) $ and $\mathbb{T}_\Upsilon(V)\otimes_{R,\varpi}E\cong V\otimes_L E(\varpi^{-1})=:V(\varpi^{-1})$ as $(G_F,E)$-modules,
  \item $D_{\dR,F}(E(\varpi^{-1}))=(\widehat{o_{\mathfrak{F}_\infty}} {\otimes}_{o_L} E(\varpi^{-1}))^{\Upsilon}$,
  \item  $  S_{\mathfrak{F}_\infty/F} \otimes_{\Lambda_{o_L}(\Upsilon),\varpi}E\cong  D_{\dR,F}(E(\varpi^{-1})) $, and
  \item $\left(D_{\dR,F}(V)\otimes_{o_F} S_{\mathfrak{F}_\infty/F} \right)\otimes_{R,\varpi}E\cong  D_{\dR,F}(V(\varpi^{-1})) $.
\end{enumerate}
\end{lemma}

\begin{proof}
 The first assertion is easy to check. For (ii) and (iii), we assume without loss of generality that $E=F$ and take $\Upsilon$-invariants of the exact sequence
 %\footnote{\color{red} Do we need for the surjectivity $H^1(\Upsilon,\widehat{o_{\mathfrak{F}_\infty}} \hat{\otimes}_{o_L}\Lambda_{o_L}(\Upsilon))=0$??? The isomorphism follows by dimension reasons!}
 \[0\to \widehat{o_{\mathfrak{F}_\infty}} \hat{\otimes}_{o_L}\mathbb{T}_\Upsilon \xrightarrow{\id\otimes \left([\varphi_{q'}]-\varpi(\varphi_{q'})\right)}\widehat{o_{\mathfrak{F}_\infty}} \hat{\otimes}_{o_L}\mathbb{T}_\Upsilon \xrightarrow{\id\otimes\varpi}     \widehat{o_{\mathfrak{F}_\infty}} {\otimes}_{o_L} E(\varpi^{-1})\to 0 \]
 with
 \[\widehat{o_{\mathfrak{F}_\infty}} \hat{\otimes}_{o_L}\mathbb{T}_\Upsilon \otimes_{\Lambda_{o_L}(\Upsilon)}\Lambda_{o_L}(\Upsilon)\xrightarrow{\id\otimes\varpi} \widehat{o_{\mathfrak{F}_\infty}} \hat{\otimes}_{o_L}\mathbb{T}_\Upsilon \otimes_{\Lambda_{o_L}(\Upsilon),\varpi}E  \cong  \widehat{o_{\mathfrak{F}_\infty}} {\otimes}_{o_L} E(\varpi^{-1}) ,\] which - by Lemma \ref{lem:Ups-cohomology} (ii) - provides a surjection
 \[S_{\mathfrak{F}_\infty/F} \twoheadrightarrow  D_{\dR,F}(E(\varpi^{-1}))\]
 upon noting that $D_{\dR,F}(E(\varpi^{-1}))=(\widehat{o_{\mathfrak{F}_\infty}} {\otimes}_{o_L} E(\varpi^{-1}))^{\Upsilon}$, because $\varpi^{-1}$ is unramified. The induced surjection
 \[S_{\mathfrak{F}_\infty/F}\otimes_{\Lambda_{o_L}(\Upsilon),\varpi}E\twoheadrightarrow  D_{\dR,F}(E(\varpi^{-1}))\]
 is injective because source and target are $F$-vector spaces of the same dimension by Lemma \ref{lem:Ups-cohomology} (ii) again.
 For (iii) use the compatibility of $D_{\dR}$ with tensor products:
 \[D_{\dR,F}(V)\otimes_{F}D_{\dR,F}(E(\varpi^{-1}))\cong D_{\dR,F}(V(\varpi^{-1})). \]
\end{proof}

We denote the composite
\begin{multline*}
  D_{\dR,F}(V)\otimes_{o_F} S_{\mathfrak{F}_\infty/F} \cong\\
  \left(D_{\dR,F}(V)\otimes_{o_F} S_{\mathfrak{F}_\infty/F} \right)\otimes_{R}R\xrightarrow{\id\otimes\varpi}\left(D_{\dR,F}(V)\otimes_{o_F} S_{\mathfrak{F}_\infty/F} \right)\otimes_{R,\varpi}E\cong  D_{\dR,F}(V(\varpi^{-1}))
\end{multline*}
by $\varpi_*.$
Similarly, using the homomorphism
\[\id_{\mathbb{T}_\Upsilon(V)}\otimes\varpi:\mathbb{T}_\Upsilon(V)\cong\mathbb{T}_\Upsilon(V)\otimes_{R}R \to \mathbb{T}_\Upsilon(V)\otimes_{R,\varpi}E,\]
also on cohomology we denote by $\varpi_*$   the composite
\[ \HIW^1(\mathfrak{F}_\infty/F, V)\cong H^1(F,  \mathbb{T}_\Upsilon(V)) \xrightarrow{H^1(F,\id_{\mathbb{T}_\Upsilon(V)}\otimes\varpi)} H^1(F,  V(\varpi^{-1})).\]

Now, by the functoriality of the above construction we obtain a commutative diagram
\begin{equation}\label{f:ExpDescent}
   \xymatrix{
     D_{\dR,F}(V)\otimes_{o_F} S_{\mathfrak{F}_\infty/F} \ar[d]_{\varpi_*} \ar[r]^{{\mathbb{E}\mathrm{xp}}_{\mathfrak{F}_\infty/F,V}} & \HIW^1(\mathfrak{F}_\infty/F, V)\ar[d]^{\varpi_*} \\
    D_{\dR,F}(V(\varpi^{-1})) \ar[r]^{\exp_{F, V(\varpi^{-1})}} & H^1(F,  V(\varpi^{-1})),  }
\end{equation}
i.e., the big exponential map ${\mathbb{E}\mathrm{xp}}_{\mathfrak{F}_\infty/F,V}$ interpolates the Bloch-Kato exponential maps $\exp_{F, V(\varpi^{-1})}$ for all unramified characters.

 It is well known, that the cup product is compatible with specialization at a character, i.e., for all $\varpi$ as above we have a commutative diagram
 \begin{equation}\label{f:Iw-pairings}
\xymatrix{
H^1(F,V^*(1)(\varpi))\ar@{}[r]|{\times} &H^1(F,V(\varpi^{-1}))  \ar[rr]^-{\langle ,\rangle_{\Tate,L}} && H^2(F,E(1))\cong E  \\
    \HIW^1(\mathfrak{F}_\infty/F, V^*(1))\ar[u]_{(\varpi^{-1})_*} \ar@{}[r]|{\times} & \HIW^1(\mathfrak{F}_\infty/F, V)\ar[u]_{\varpi_* } \ar[rr]^-{\langle ,\rangle_{\Tate,R} } && R.\ar[u]_{ \varpi_*} }
\end{equation}

Similarly,   the usual $D_{\dR}$-pairing extends to a $R$-$\upiota_*$-sesquilinear %\com{\color{red} zur Zeit $\iota_*$-linear im ersten Argument, vielleicht muss das vertauscht werden!}
pairing satisfying  analogous properties:
 \small
\begin{equation}\label{f:dR-pairings}
\xymatrix@C=-2pt{
 D_{\dR,F}(V^*(1)(\varpi))&{\times} &  D_{\dR,F}(V(\varpi^{-1})) \ar[rrr]^-{[,]_{\dR}^{}} &&& D_{\dR,F}(E(1))\cong E\phantom{mmmmmmmm}  \\
   D_{\dR,F}(V^*(1))\otimes_{o_L} S_{\mathfrak{F}_\infty/F} \ar[u]_{(\varpi^{-1})_*} &{\times} &  D_{\dR,F}(V)\otimes_{o_F} S_{\mathfrak{F}_\infty/F}\ar[u]_{\varpi_* } \ar[rrr]^-{[,]^0}&&& D_{\dR,F}(L(1))\otimes_{o_L} \Lambda_{o_L}(\Upsilon)\cong R. \phantom{mmmmmmmmmmmmm} \ar[u]_{\varpi_*} }
\end{equation}\normalsize
where the bottom pairing arises as follows: Note that the free rank one $\Lambda_{o_L}(\Upsilon)^{\mathrm{op}}$-module $ S_{\mathfrak{F}_\infty/F}^*:=\Hom_{\Lambda_{o_L}(\Upsilon)}(S_{\mathfrak{F}_\infty/F},\Lambda_{o_L}(\Upsilon))$ can be identified with
\begin{equation}\label{f:YagerDual}
   \iota( S_{\mathfrak{F}_\infty/F})
=\{ \lambda\in  \widehat{o_{\mathfrak{F}_\infty}}\hat{\otimes}_{o_L}\Lambda_{o_L}(\Upsilon)\mid (u\otimes \id)(\lambda)=(\id\otimes [u^{-1}])\lambda\}
\end{equation}
for any topological generator $u$ of $\Upsilon.$ In particular, the duality is given by the multiplication
\[S_{\mathfrak{F}_\infty/F}^* \times S_{\mathfrak{F}_\infty/F}\mapsto \Lambda_{o_L}(\Upsilon).\]
As usual in Iwasawa theory we view it as $\Lambda_{o_L}(\Upsilon)$-module via the involution $\Lambda_{o_L}(\Upsilon) \xrightarrow[\iota]{\cong}\Lambda_{o_L}(\Upsilon)^{\mathrm{op}}.$
Then, $(S_{\mathfrak{F}_\infty/F}^*)^\iota$ is isomorphic to $ S_{\mathfrak{F}_\infty/F}$ and the bottom pairing in \eqref{f:dR-pairings} sends $(d'\otimes s', d\otimes s)$ to $[d',d]\otimes \iota(s')s.$

%{\color{red}
%In particular, the duality is given by the multiplication
%\[S_{\mathfrak{F}_\infty/F}  \times S_{\mathfrak{F}_\infty/F}\mapsto \Lambda_{o_L}(\Upsilon)\]
%and the bottom pairing in \eqref{f:dR-pairings} sends $(d'\otimes s', d\otimes s)$ to $[d',d]\otimes \iota(s')s.$
%}

%For the rest of this section we extend our standard assumption to\\
%
%{\bf Assumption:} $H^0(G_{\mathfrak{F}_\infty},T)=0=H^0(G_{\mathfrak{F}_\infty},T^*(1))$.\\

Upon noting that the bottom pairing  of
%\eqref{f:Iw-pairings} (by Proposition \ref{prop:Iw-perfectDual} and the Assumption) and
\eqref{f:dR-pairings} is non-degenerate,   we can now define the big dual exponential map ${\mathbb{E}\mathrm{xp}}^*_{\mathfrak{F}_\infty/F,V^*(1)}$ as adjoint of ${\mathbb{E}\mathrm{xp}}_{\mathfrak{F}_\infty/F,V}$, i.e., via the commutativity of the following diagram
\begin{equation}\label{f:adjointness}
   \xymatrix{
D_{\dR,F}(V^*(1))\otimes_{o_L} S_{\mathfrak{F}_\infty/F}  \ar@{}[r]|{\times} &  D_{\dR,F}(V)\otimes_{o_F} S_{\mathfrak{F}_\infty/F}  \ar[d]^{{\mathbb{E}\mathrm{xp}}_{\mathfrak{F}_\infty/F,V}} \ar[rr]^-{[,]^0}&& R   \\
    \HIW^1(\mathfrak{F}_\infty/F, V^*(1))\ar@{-->}[u]_{{\mathbb{E}\mathrm{xp}}^*_{\mathfrak{F}_\infty/F,V^*(1)}} \ar@{}[r]|{\times} & \HIW^1(\mathfrak{F}_\infty/F, V)   \ar[rr]^-{\langle ,\rangle_{\Tate,R} } && R.\ar@{=}[u]_{  }}
\end{equation}
In other words, ${\mathbb{E}\mathrm{xp}}^*_{\mathfrak{F}_\infty/F,V^*(1)}$  is the composite\footnote{Note that we do not need that the first map is an isomorphism, i.e.,  perfectness of the Tate pairing for Iwasawa cohomology is not required.}
\begin{multline*}
  \HIW^1(\mathfrak{F}_\infty/F, V^*(1))\xrightarrow{x\mapsto \langle x,-\rangle} \HIW^1(\mathfrak{F}_\infty/F, V)^* \xrightarrow{({\mathbb{E}\mathrm{xp}}_{\mathfrak{F}_\infty/F,V})^*}\\
 (D_{\dR,F}(V)\otimes_{o_F} S_{\mathfrak{F}_\infty/F}  )^* \xleftarrow[{\cong}]{d\mapsto [d,-]^0}D_{\dR,F}(V^*(1))\otimes_{o_L} S_{\mathfrak{F}_\infty/F}.
\end{multline*}
%in which the exterior maps are isomorphisms.
From \eqref{f:ExpDescent}, \eqref{f:adjointness}, \eqref{f:Iw-pairings} and \eqref{f:dR-pairings} we obtain the following specialization property
\begin{equation}\label{f:dualExpDescent}
   \xymatrix{
      \HIW^1(\mathfrak{F}_\infty/F, V^*(1))\ar[d]^{(\varpi^{-1})_*} \ar[rrr]^{{\mathbb{E}\mathrm{xp}}^*_{\mathfrak{F}_\infty/F,V^*(1)}} &&& D_{\dR,F}(V^*(1))\otimes_{o_L} S_{\mathfrak{F}_\infty/F} \ar[d]^{(\varpi^{-1})_*}\\
   H^1(F,V^*(1)(\varpi))  \ar[rrr]^{\exp^*_{F, V^*(1)(\varpi)}} &&& D_{\dR,F}(V^*(1)(\varpi))  ,  }
\end{equation}
i.e., the big dual exponential map ${\mathbb{E}\mathrm{xp}}^*_{\mathfrak{F}_\infty/F,V^*(1)}$ interpolates the dual Bloch-Kato exponential maps. Indeed, the mentioned compatibilities imply that for fixed $x\in \HIW^1(\mathfrak{F}_\infty/F, V^*(1))$ and for all $d\in  D_{\dR,F}(V)\otimes_{o_F} S_{\mathfrak{F}_\infty/F}$
\begin{align*}
  [(\varpi^{-1})_*\circ  {\mathbb{E}\mathrm{xp}}^*_{\mathfrak{F}_\infty/F,V^*(1)}(x),\varpi_*(d)]_{dR}& =\varpi_*\left([  {\mathbb{E}\mathrm{xp}}^*_{\mathfrak{F}_\infty/F,V^*(1)}(x),d ]^0 \right) \\
    & =\varpi_*\left(\langle  x,{\mathbb{E}\mathrm{xp}}_{\mathfrak{F}_\infty/F,V}(d) \rangle_{\Tate,R} \right)  \\
    &  =\langle (\varpi^{-1})_* (x),\varpi_*({\mathbb{E}\mathrm{xp}}_{\mathfrak{F}_\infty/F,V}(d)) \rangle_{\Tate,E} \\
     &  =\langle (\varpi^{-1})_* (x),\exp_{F, V(\varpi^{-1})}(\varpi_*(d)) \rangle_{\Tate,E} \\
      &  =\langle \exp^*_{F, V^*(1)((\varpi)}(\varpi^{-1})_* (x)), \varpi_*(d) \rangle_{\Tate,E}, \\
\end{align*}
 where the last equality stems from Kato's characterization of the dual Bloch-Kato exponential map. As $\varpi_*(d)$ runs through all elements of $D_{\dR,F}(V(\varpi^{-1})), $ the commutativity of \eqref{f:dualExpDescent} follows.

Note that we have defined a big dual exponential map ${\mathbb{E}\mathrm{xp}}^*_{\mathfrak{F}_\infty/F,V }$ only under the condition that  $H^0(G_{\mathfrak{F}_\infty},V^*(1))$  vanishes. This condition grants the existence of ${\mathbb{E}\mathrm{xp}}_{\mathfrak{F}_\infty/F,V^*(1) }$ and we used the above pairings for the definition of its dual. Now we want to pursue a way of direct construction of the big dual exponential map which is closer to Kato's definition of  $\exp^*_{F, V}$ rather than its property how it is related to the literal dual of  $\exp_{F, V^*(1)}$. Henceforth we stick literally to our above Assumption \ref{assump} when constructing a second big dual exponential map ${\mathfrak{E}\mathrm{xp}}^*_{\mathfrak{F}_\infty/F,V }$, but  note that it is the dual assumption as for ${\mathbb{E}\mathrm{xp}}^*_{\mathfrak{F}_\infty/F,V }$.

We still assume, in addition,   that $V$ is de Rham and we consider now the dual Bloch-Kato exponential map from \cite[\S 1.2.4]{Kat}
\begin{equation}\label{f:BKFinfty}
   \exp^*_{\mathfrak{F}_\infty,V}: H^1(\mathfrak{F}_\infty,V)\to H^1(\mathfrak{F}_\infty,V\otimes_{\qp} B^+_\dR) \xleftarrow[\cong]{-\cup\log(\chi_{\rm cyc})} H^0(\mathfrak{F}_\infty,V\otimes_{\qp} B^+_\dR)=D^0_{\dR,\mathfrak{F}_\infty}(V),
\end{equation}
where $\log(\chi_{\rm cyc}):G_{\mathfrak{F}_\infty}\to\zp $ is considered as $1$-cocycle in $H^1(\mathfrak{F}_\infty,\zp).$
Completed   base change   with respect to $\mathbb{T}_\Upsilon$ gives
\begin{multline}\label{f:basechangeddual}
  \exp^*_{\mathfrak{F}_\infty,V}\hat{\otimes}_{o_L}\mathbb{T}_\Upsilon: H^1(\mathfrak{F}_\infty,V)\hat{\otimes}_{o_L}\mathbb{T}_\Upsilon\to H^1(\mathfrak{F}_\infty,V\otimes_{\qp} B^+_\dR)\hat{\otimes}_{o_L}\mathbb{T}_\Upsilon \\ \xleftarrow[\cong]{(-\cup\log\chi_{\rm cyc})\hat{\otimes}_{o_L}\mathbb{T}_\Upsilon } D^0_{\dR,\mathfrak{F}_\infty}(V)\hat{\otimes}_{o_L}\mathbb{T}_\Upsilon.
\end{multline}
By Lemma \ref{lem:completedbasechange} we can identify $H^1(\mathfrak{F}_\infty,V)\hat{\otimes}_{o_L}\mathbb{T}_\Upsilon= H^1(\mathfrak{F}_\infty,V{\otimes}_{o_L}\mathbb{T}_\Upsilon)$. The same arguments show the analogous statement
\[H^1(\mathfrak{F}_\infty,V\otimes_{\qp} B^+_\dR)\hat{\otimes}_{o_L}\mathbb{T}_\Upsilon \cong H^1(\mathfrak{F}_\infty,V\otimes_{\qp} B^+_\dR{\otimes}_{o_L}\mathbb{T}_\Upsilon) \]
upon noting that this time the term
\[R^1 \varprojlim_U H^{0}_{}(\mathfrak{F}_\infty, V\otimes_{\qp} B^+_\dR{\otimes}_{o_L}(\mathbb{T}_\Upsilon)_U)\cong R^1 (\varprojlim_U D_{\dR,\mathfrak{F}_\infty}(V) {\otimes}_{o_L}(\mathbb{T}_\Upsilon)_U)\]
vanishes, because the transitions map are visibly surjective. Taking $\Upsilon$-invariants together with \eqref{f:DdR} and \eqref{f:Sinftyproperty} we thus arrive at
\begin{multline}\label{f:defExpfinal}
  {\mathfrak{E}\mathrm{xp}}^\ast_{\mathfrak{F}_\infty/F,V}: \HIW^1(\mathfrak{F}_\infty/F, V)\cong  H^1(F, \mathbb{T}_\Upsilon(V))\to H^1(F, \mathbb{T}_\Upsilon(V\otimes_{\qp} B^+_\dR))\\ \xleftarrow[\cong]{\left((-\cup\log\chi_{\rm cyc})\hat{\otimes}_{o_L}\mathbb{T}_\Upsilon \right)^\Upsilon} D^0_{\dR,F}(V)\otimes_{o_F} S_{\mathfrak{F}_\infty/F}   ,
\end{multline}
where we again use \eqref{f:HS} for the first two terms upon noting that
\begin{align*}
   H^1(\Upsilon,H^0(\mathfrak{F}_\infty,V\otimes_{\qp} B^+_\dR{\otimes}_{o_L}\mathbb{T}_\Upsilon)) & \cong H^1(\Upsilon, D^0_{\dR,\mathfrak{F}_\infty}(V)\hat{\otimes}_{o_L}\mathbb{T}_\Upsilon) \\
    & \cong H^1(\Upsilon, D^0_{\dR,F}(V){\otimes}_{o_F} \widehat{o_{\mathfrak{F}_\infty}} \hat{\otimes}_{o_L}\mathbb{T}_\Upsilon)=0
  \end{align*}
by Lemma \ref{lem:Ups-cohomology}(ii). By the functoriality of cohomology and in particular of the cup-product we obtain immediately the commutative descent diagram
\begin{equation}\label{f:dualExpDescent}
   \xymatrix{
      \HIW^1(\mathfrak{F}_\infty/F, V )\ar[d]^{(\varpi^{-1})_*} \ar[rrr]^{{\mathfrak{E}\mathrm{xp}}^*_{\mathfrak{F}_\infty/F,V }} &&& D_{\dR,F}(V )\otimes_{o_L} S_{\mathfrak{F}_\infty/F} \ar[d]^{(\varpi^{-1})_*}\\
   H^1(F,V (\varpi))  \ar[rrr]^{\exp^*_{F, V (\varpi)}} &&& D_{\dR,F}(V (\varpi))  ,  }
\end{equation}
for any character $\varpi$ of $\Upsilon.$ In case that $H^0(G_{\mathfrak{F}_\infty},V)$ and $H^0(G_{\mathfrak{F}_\infty},V^*(1))$ both vanish one easily checks that ${\mathfrak{E}\mathrm{xp}}^*_{\mathfrak{F}_\infty/F,V }$ and ${\mathbb{E}\mathrm{xp}}^*_{\mathfrak{F}_\infty/F,V }$ coincide as they satisfy the same interpolation property.

{\bf Convention:} In order to simplify we henceforth we also use the notation ${\mathbb{E}\mathrm{xp}}^*_{\mathfrak{F}_\infty/F,V }$ for ${\mathfrak{E}\mathrm{xp}}^*_{\mathfrak{F}_\infty/F,V }$, i.e., with this convention ${\mathbb{E}\mathrm{xp}}^*_{\mathfrak{F}_\infty/F,V }$ whenever   $H^0(G_{\mathfrak{F}_\infty},V)$ or $H^0(G_{\mathfrak{F}_\infty},V^*(1))$ vanishes.

Moreover, one could be tempted to try to use an analogue of diagram \eqref{f:adjointness} to define a big exponential map ${\mathfrak{E}\mathrm{xp}}_{\mathfrak{F}_\infty/F,V^*(1) }$ under our standard Assumption \ref{assump}. But unfortunately, for this one would require an isomorphism
\[ \HIW^1(\mathfrak{F}_\infty/F, V^*(1))\xrightarrow[\cong]{x\mapsto \langle x,-\rangle} \HIW^1(\mathfrak{F}_\infty/F, V)^* \]
 which is only granted by Proposition \ref{prop:Iw-perfectDual}, if also $H^0(G_{\mathfrak{F}_\infty},V^*(1))$ vanishes, i.e., in the case when ${\mathbb{E}\mathrm{xp}}_{\mathfrak{F}_\infty/F,V^*(1) }$ exists anyway. Only, if one is willing to forget the $\Lambda_{o_L}(\Upsilon)$-torsion in $\HIW^1(\mathfrak{F}_\infty/F, V^*(1))$ one can construct ${\mathfrak{E}\mathrm{xp}}_{\mathfrak{F}_\infty/F,V^*(1) }$ in new cases.

\com{\color{red} formulate versions for the id-components!}

{\bf Question:} This approach is based on the assumption that $H^0(G_{\mathfrak{F}_\infty},V)$  (or $H^0(G_{\mathfrak{F}_\infty},V^*(1))$) vanishes. One might want to consider also the other extreme that $H^0(G_{\mathfrak{F}_\infty},V)=V$ (or $H^0(G_{\mathfrak{F}_\infty},V^*(1))=V^*(1)$) or the general case without any such condition. Does there also exist a reasonable big (dual) exponential map in those cases?

\subsection{Colmez' reciprocity law}\label{sec:Colmez}

The purpose of this subsection is to improve the following theorem from \cite[Thm.\ 3.2.1]{SanoV}. We assume $F=L$ throughout this subsection.

Let   $T$ be in $\Rep_{o_L,f}^{\cris}(G_L)$    such that $T(\tau^{-1})$ belongs to  $\Rep_{o_L,f}^{\cris,\mathrm{an}}(G_L)$ with all   Hodge--Tate weights   in $[1,r],$ and such that $V:=L\otimes_{o_L} T$ does not have any quotient isomorphic to $L(\tau).$ Then the assumptions of \cite[Lem.\ 3.3.6]{SV24}  are satisfied and we may define the  maps
$  \varphi_L^{-m}:D_{\LT}(T(\tau^{-1}))^{\psi_L=1}\to L_m[[t_{\rm LT}]]\otimes_L D_{\cris,L}(V(\tau^{-1})) $   via the  inclusions
%\[(\mathbf{A}_L(\chi_{\rm LT}^{1-j}))^{\psi=1}=(\omega_{\rm LT}^{j-2}\mathbf{A}_L^+(\chi_{\rm LT}^{1-j}))^{\psi=1}\]
 - by \cite[(3.16), Cor.\ 3.1.14]{SV24} involving  Kisin's and Ren's
comparison isomorphism -,
\begin{equation}\label{f:compWachDcris}
    D_{\LT}(T(\tau^{-1}))^{\psi_L=1}=(N(T(\tau^{-1})))^{\psi=1}\subseteq N(T(\tau^{-1}))\subseteq \cR_L^+\otimes_L D_{\cris,L}(V(\tau^{-1}))
\end{equation}
and  the map $\varphi_q^{-m}$ defined before \cite[Lem.\ 2.4.3]{BF}.

\begin{theorem}[Colmez' reciprocity law in the Lubin-Tate setting]\label{thm:Colmez} In the above setting assume in addition  that $V$ satisfies $D_{\cris,L}(V(\tau^{-1}))^{\varphi_L=1}=D_{\cris,L}(V(\tau^{-1}))^{\varphi_L=\pi_L^{-1}}=0$. Then, for all $j\geq 1$   such that the operators  $1-\pi_L^{-1-j}\varphi_L^{-1}, 1-\frac{\pi_L^{j+1}}{q}\varphi_L $ are invertible on $D_{\cris,L}(V(\tau^{-1}))$, the following diagram commutes:
\begin{equation}\label{f:interpolationgeneral}
 \xymatrix{
   H^1_{\Iw}(L_\infty{/L},T) \ar[dd]_{\mathrm{pr}_{L_\infty/L_m}^{-j}} \ar[rr]^-{\mathrm{Exp}_{L_\infty/L,T}^*} &&  D_{\LT}(T(\tau^{-1}))^{\psi_L=1}  \ar[d]^{\pi_L^{-m}\varphi_L^{-m}\otimes \mathbf{d}_1}   \\
   &&   L_m((t_{\rm LT}))\otimes_L D_{\cris,L}(V) \ar[d]^{c_{t_{\rm LT}^{j}}\otimes (\id\otimes e_{-j})} \\
   H^1(L_m,V(\chi_\mathrm{\rm LT}^{-j}))\ar[rr]^-{\exp_{L_m,V(\chi_\LT^{-j})}^*}&&   L_m \otimes_L D_{\cris,L}(V(\chi_\LT^{-j})).  }
\end{equation}
In other words, in $L_m((t_\LT))\otimes_L D_{\mathrm{dR},L}(V) \subseteq B_\dR \otimes_L V$ we have
\[\pi_L^{-m}\varphi_L^{-m}(\mathrm{Exp}^*)\otimes \mathbf{d}_1=\sum_{j\geq 1 } \exp^*_{L_m,V(\chi_\LT^{-j})}\circ \mathrm{pr}_{L_\infty/L_m}^{-j}=\sum_{j\in\mathbb{Z} }\exp^*_{L_m,V(\chi_\LT^{-j})}\circ \mathrm{pr}_{L_\infty/L_m}^{-j},\]
if the above conditions hold for all $j\geq 1,$ i.e., if $\varphi_L$ acting on $D_{\cris,L}(V(\tau^{-1}))$ does not have any eigenvalue in $\pi_L^{\mathbb{N}_0}\cup q\pi_L^{\mathbb{N}_0};$ to this end we identify the target $L_m \otimes_L D_{\cris,L}(V(\chi_\LT^{-j}))$ with $L_mt_\LT^{j}\otimes_L D_{\cris,L}(V)\subseteq L_m((t_\LT))\otimes_L D_{\mathrm{dR},L}(V).$ Note that $c_{t_\LT^i}(\pi_L^{-m}\varphi_L^{-m}(\mathrm{Exp}^*)\otimes \mathbf{d}_1)=0$ for $i\leq 0$, because ${\rm Fil}^0 D_{\dR,L}(V)=0$ as the Hodge-Tate weights of $V$ are at least $1.$
\end{theorem}

By construction of \cite[Lem.\ 6.6.1]{Ku} we have a commutative diagram, for $T$ a {\em good } crystalline representation of $G_L$, i.e., those satisfying the conditions of Theorem \ref{thm:Kupferer},\footnote{By an analogue of \cite[Prop.\ 4.5]{LZ} we do not need the assumption concerning trivial quotients, if $\mathfrak{F}_\infty/L$ is infinite!}
\begin{equation}\label{f:descentXtoL}
\xymatrix{
  H^1_{\mathrm{Iw}}(X_\infty/L,V) \ar[d]_{\mathrm{pr}_{X_\infty/L_\infty}^\varpi} \ar[rr]^-{\mathrm{Exp}^*_{X_\infty/L,V}} & & \mathcal{N}_{\mathfrak{F}_\infty/L}(V(\tau^{-1}))^{\psi=1}\cong \left(\mathcal{N}_{L/L}(V(\tau^{-1}))\hten_{o_L} S_{\mathfrak{F}_\infty/L}\right)^{\psi=1}\ar[d]^{(\varpi^{-1})_*}  \\
  H^1_{\mathrm{Iw}}({L}_\infty /L,V(\varpi)) \ar[rr]^-{\mathrm{Exp}^*_{L_\infty/L,V(\varpi)}} && \mathcal{N}_{L_\infty/L}(V(\tau^{-1})(\varpi))^{\psi=1}\cong D_{\LT}(V(\tau^{-1})(\varpi))^{\psi_L=1}.   }
\end{equation}

Now we are interested in the descent to the unramified part, consider the following diagram for $j\geq 1$, whenever $ \mathbb{E}\mathrm{xp}^*_{L_m\mathfrak{F}_\infty/L_m,V(\chi_\LT^{-j})}$   is defined, i.e., if $H^0(G_{L_m\mathfrak{F}_\infty},V^*(\chi_\LT^{j})(1))=0$\footnote{Note that for $j\gg1$ the Hodge-Tate weights of $V^*(\chi_\LT^{j})(1) $ are all positive. Then $D_{\dR}^0$ of it is zero (over $L$, but then also over $L_m\mathfrak{F}_\infty$ by faithfully flat base change), whence $H^0(G_{L_m\mathfrak{F}_\infty},V^*(\chi_\LT^{j})(1))=0$ by the fundamental exact sequence, which implies that $H^0(G_{L_m\mathfrak{F}_\infty},W)\subseteq D_{\dR,L_m\mathfrak{F}_\infty}^0(W)$ for general $W$. } or $H^0(G_{L_m\mathfrak{F}_\infty},V(\chi_\LT^{-j}))=0$.

\begin{equation}\label{f:interpolationunramified}
 \xymatrix{
   H^1_{\Iw}(X_\infty/L,V) \ar[dd]_{\mathrm{pr}_{X_\infty/L_m\mathfrak{F}_\infty}^{-j}} \ar[rr]^-{\mathrm{Exp}^*_{X_\infty/L,V}} && \mathcal{N}_{\mathfrak{F}_\infty/L}(V(\tau^{-1}))^{\psi=1}\cong \left(\mathcal{N}_{L/L}(V(\tau^{-1}))\hten_{o_L} S_{\mathfrak{F}_\infty/L}\right)^{\psi=1} \ar[d]^{\pi_L^{-m}\varphi_L^{-m}\otimes \mathbf{d}_1}   \\
   &&   L_m((t_{\rm LT}))\hten_{o_L} S_{\mathfrak{F}_\infty/L}\otimes_L D_{\cris,L}(V) \ar[d]^{c_{t_{\rm LT}^{j}}\otimes (\id\otimes e_{-j})} \\
   H^1_{\Iw}(L_m\mathfrak{F}_\infty/L_m,V(\chi_\mathrm{\rm LT}^{-j}))\ar[rr]^-{{\mathbb{E}\mathrm{xp}}^*_{L_m\mathfrak{F}_\infty/L_m,V(\chi_\LT^{-j})}}&&   L_m \otimes_{o_L} S_{\mathfrak{F}_\infty/L} \otimes_L D_{\cris,L}(V(\chi_\LT^{-j})).  }
\end{equation}
Here, the right vertical maps are just defined using base change with $-\hten_{o_L} S_{\mathfrak{F}_\infty/L}$ (with respect to the first one before taking $\psi$ invariants, i.e., regarding the last inclusion of \eqref{f:compWachDcris}). In order to analyze the commutativity of this diagram we first show the following lemma.
\begin{lemma}\label{lem:relative} The diagram
 \begin{equation}\label{f:relativedR}\footnotesize
    \xymatrix{
      \left(\mathcal{N}_{L/L}(V(\tau^{-1}))\hten_{o_L} S_{\mathfrak{F}_\infty/L}\right)^{\psi=1} \ar[d]_{\left(c_{t_{\rm LT}^{j}}\otimes (\id\otimes e_{-j})\right)\circ\pi_L^{-m}\varphi_L^{-m}(-)\otimes \mathbf{d}_1} \ar[rr]^-{(\varpi^{-1})_*} &  & \mathcal{N}_{L_\infty/L}(V(\tau^{-1})(\varpi))^{\psi=1} \ar[d]^{\left(c_{t_{\rm LT}^{j}}\otimes (\id\otimes e_{-j})\right)\circ\pi_L^{-m}\varphi_L^{-m}(-)\otimes \mathbf{d}_1} \\
     L_m \otimes_{o_L} S_{\mathfrak{F}_\infty/L} \otimes_L D_{\cris,L}(V(\chi_\LT^{-j})) \ar[rr]^-{(\varpi^{-1})_*} & & L_m\otimes_L D_{\dR,L}(V(\chi_\LT^{-j})(\varpi))   }
 \end{equation}
 commutes.
\end{lemma}
\begin{proof}
  Upon unwinding the definitions we have to show the commutativity of the diagram
\begin{equation*}
  \xymatrix{
    \mathcal{N}_{L/L}(V(\tau^{-1}))\otimes_{L} (\widehat{o_{\mathfrak{F}_\infty}} {\otimes}_{o_L} E(\varpi^{-1}))^{\Upsilon} \ar[d]_{ } \ar[r]^{ } & \cR_L^+\otimes_L D_{\cris,L}(V(\tau^{-1}))\otimes_{L} (\widehat{o_{\mathfrak{F}_\infty}} {\otimes}_{o_L} E(\varpi^{-1}))^{\Upsilon}\ar[d]^{ } \\
    \mathcal{N}_{L/L}(V(\tau^{-1})(\varpi))\ar[r]^{ } &  \cR_L^+\otimes_L D_{\cris,L}(V(\tau^{-1})(\varpi)) ,  }
\end{equation*}
in which the top line arises as base  change with $- \otimes_{L} (\widehat{o_{\mathfrak{F}_\infty}} {\otimes}_{o_L} E(\varpi^{-1}))^{\Upsilon}$ of the last inclusion of \eqref{f:compWachDcris} while the bottom line is the analogous inclusion for $V(\tau^{-1})(\varpi).$ But since, by definition upon using   Lemma \ref{lem:KR-modules} (v) and Lemma \ref{lem:varpidescent} (ii),  the vertical maps are the canonical isomorphisms expressing the compatibilities of the functors $\mathcal{N}_{L/L}$  and $D_{\cris,L}$ with tensor-products, the claim follows from the fact that the comparison isomorphisms (induced by the horizontal inclusions) are also compatible with tensor products, see \cite[(3.13)]{SV24}.
\end{proof}

We will show the commutativity of \eqref{f:interpolationunramified} (for fixed $j$) by specializing at infinitely many points (which suffices as a non-zero power series in $R$ only can have finitely many zeroes by the Weierstrass Preparation Theorem). In order to calculate these specializations the reader is invited to look at diagram \eqref{f:hugediag} below. The specialization of \[ {\mathbb{E}\mathrm{xp}}^*_{L_m\mathfrak{F}_\infty/L_m,V(\chi_\LT^{-j})}\circ \mathrm{pr}_{X_\infty/L_m\mathfrak{F}_\infty}^{-j}\] at $\varpi$ coincides with $\exp^*_{L_m, V(\chi_\LT^{-j})(\varpi)}\circ  \mathrm{pr}_{X_\infty/L_m}^\varpi$ by \eqref{f:dualExpDescent}, which in turn can be expressed as
\[\left(c_{t_{\rm LT}^{j}}\otimes (\id\otimes e_{-j})\right)\circ\pi_L^{-m}\varphi_L^{-m}(\mathrm{Exp}^*_{L_\infty/L,V(\varpi)})\otimes \mathbf{d}_1\circ\mathrm{pr}_{X_\infty/L_\infty}^\varpi\]
by Theorem \ref{thm:Colmez} for those $\varpi$, such that for $V(\varpi)$ the conditions of that theorem are satisfied, i.e.,  $D_{\cris,L}(V(\tau^{-1}\varpi))^{\varphi_L=1}=D_{\cris,L}(V(\tau^{-1}\varpi))^{\varphi_L=\pi_L^{-1}}=0$ and   the operators  $1-\pi_L^{-1-j}\varphi_L^{-1}, 1-\frac{\pi_L^{j+1}}{q}\varphi_L $ are invertible on $D_{\cris,L}(V(\tau^{-1}\varpi))$. Note that this holds for all but finitely many $\varpi.$
Using   \eqref{f:descentXtoL} this map equals
\[\left(c_{t_{\rm LT}^{j}}\otimes (\id\otimes e_{-j})\right)\circ\pi_L^{-m}\varphi_L^{-m}((\varpi^{-1})_\ast \circ\mathrm{Exp}_{X_\infty/L,V}^*)\otimes \mathbf{d}_1\]
Finally, the commutative diagram \eqref{f:relativedR}
 shows that the latter map coincides with the specialization of \[\left(c_{t_{\rm LT}^{j}}\otimes (\id\otimes e_{-j})\right)\circ \pi_L^{-m}\varphi_L^{-m}(\mathrm{Exp}^*_{X_\infty/L,V} )\otimes \mathbf{d}_1 \] at $\varpi$ as claimed. Hence we have shown the following result.

 \begin{theorem}\label{thm:bigColmez}
 Let $T(\tau^{-1}) \in \Rep_{o_L,f}^{\cris,\mathrm{an}}(G_L)$ and $V(\tau^{-1})=T(\tau^{-1})[1/\piL]$ with nonnegative Hodge-Tate weights.   For    $j\geq 1$ such that $H^0(G_{L_m\mathfrak{F}_\infty},V^*(\chi_\LT^{j})(1))=0$ or $H^0(G_{L_m\mathfrak{F}_\infty},V(\chi_\LT^{-j}))=0$, the diagram \eqref{f:interpolationunramified} commutes, i.e., an analogue of Colmez' reciprocity law holds for the big dual exponential maps or in other words for unramified families simultaneously.
 \end{theorem}

As a consequence we obtain the following improvement of Theorem \ref{thm:Colmez} in the sense that we get rid of some conditions (paying the price that we have possibly to restrict the source slightly).

\begin{theorem}\label{thm:ColmzezImproved}
Let $T(\tau^{-1}) \in \Rep_{o_L,f}^{\cris,\mathrm{an}}(G_L)$ and $V(\tau^{-1})=T(\tau^{-1})[1/\piL]$ with nonnegative Hodge-Tate weights. Then, on the image $I_V$ of $  H^1_{\mathrm{Iw}}(X_\infty/L,V) \xrightarrow{\mathrm{pr}_{X_\infty/L_\infty}} H^1_{\mathrm{Iw}}(L_\infty/L,V)$ diagram \eqref{f:interpolationgeneral} commutes for all $j\geq 1$, such that $H^0(G_{L_m\mathfrak{F}_\infty},V^*(\chi_\LT^{j})(1))=0$ or \linebreak $H^0(G_{L_m\mathfrak{F}_\infty},V(\chi_\LT^{-j}))=0$.
\end{theorem}

\begin{proof}
Fix $m,j\geq 1$ and  assume that $j $ satisfies $H^0(G_{L_m\mathfrak{F}_\infty},V^*(\chi_\LT^{j})(1))=0$ or \linebreak $H^0(G_{L_m\mathfrak{F}_\infty},V(\chi_\LT^{-j}))=0$.
  %and choose an (unramified) character $\varpi$ of $\Upsilon$ such that both
%  $D_{\cris,L}(V(\tau^{-1}\varpi))^{\varphi_L=1}=D_{\cris,L}(V(\tau^{-1}\varpi))^{\varphi_L=\pi_L^{-1}}=0$ and   the operators  $1-\pi_L^{-1-j}\varphi_L^{-1}, 1-\frac{\pi_L^{j+1}}{q}\varphi_L $ are invertible on $D_{\cris,L}(V(\tau^{-1}\varpi))$.
  Then $V$ satisfies the conditions of Theorem \ref{thm:Colmez}.
  %Lift a given element $y\in I_V$ to an element $x\in H^1_{\mathrm{Iw}}(X_\infty/L,V).$ For the commutativity of
Look at the diagram \eqref{f:hugediag} below with $\varpi$ being the trivial character. Instead of comparing the images of any given element $y\in I_V$ under the two different paths in the front page from the upper left to the lower right corner, we may check the images of any lift of $y$ to $x\in H^1_{\mathrm{Iw}}(X_\infty/L,V)$ under paths which replaces those from the front side using the other faces which are already known to commute. The claim follows.
%
%\hrule
%
% Assume {\color{red} (first???)} that $D_{\cris,L}(V(\tau^{-1}))^{\varphi_L=1}=D_{\cris,L}(V(\tau^{-1}))^{\varphi_L=\pi_L^{-1}}=0$ and let $j\geq 1$ and fix $m.$ In case the operators  $1-\pi_L^{-1-j}\varphi_L^{-1}, 1-\frac{\pi_L^{j+1}}{q}\varphi_L $ are invertible on $D_{\cris,L}(V(\tau^{-1}))$, the commutativity is already known by the theorem. Otherwise  lift a given element $y\in I_V$ to an element $x\in H^1_{\mathrm{Iw}}(X_\infty/L,V).$ For the commutativity of
\end{proof}

The results can be summarized by the commutativity of the following diagram    for $V=T[\frac{1}{\pi_L}]$ satisfying:
\begin{enumerate}
  \item $T(\tau^{-1}) \in \Rep_{o_L,f}^{\cris,\mathrm{an}}(G_L),$
  \item $V(\tau^{-1})$ with nonnegative Hodge-Tate weights, and
  \item $H^0(G_{L_m\mathfrak{F}_\infty},V^*(\chi_\LT^{j})(1))=0$.
\end{enumerate}

\begin{equation}\label{f:hugediag}\resizebox{\hsize}{!}{
   \xymatrix{
  & H^1_{\mathrm{Iw}}(X_\infty/L,V) \ar[rr]^-{\mathrm{Exp}^*_{X_\infty/L,V}} \ar'[d][dd]_{\mathrm{pr}_{X_\infty/L_m\mathfrak{F}_\infty}^{-j}}\ar[dl]_-{\mathrm{pr}_{X_\infty/L_\infty}^\varpi}
      &  &  \left(\mathcal{N}_{L/L}(V(\tau^{-1}))\hten_{o_L} S_{\mathfrak{F}_\infty/L}\right)^{\psi=1} \ar[dd]^{c_{t_{\rm LT}^{j}} \otimes e_{-j}\circ\pi_L^{-m}\varphi_L^{-m} \otimes \mathbf{d}_1}    \ar[dl]^-{(\varpi^{-1})_*}    \\
 H^1_{\mathrm{Iw}}(L_\infty/L,V(\varpi)) \ar@{-->}[rr]^-{\mathrm{Exp}^*_{L_\infty/L,V(\varpi)}}\ar@{-->}[dd]_{\mathrm{pr}_{L_\infty/L_m}^{-j}}
      &  &  \mathcal{N}_{L_\infty/L}(V(\tau^{-1})(\varpi))^{\psi=1} \ar[dd]^(0.65){c_{t_{\rm LT}^{j}} \otimes e_{-j}\circ\pi_L^{-m}\varphi_L^{-m} \otimes \mathbf{d}_1} \\
  & H^1_{\Iw}(L_m\mathfrak{F}_\infty/L_m,V(\chi_\mathrm{\rm LT}^{-j})) \ar'[r]^-{{\mathbb{E}\mathrm{xp}}^*_{L_m\mathfrak{F}_\infty/L_m,V(\chi_\mathrm{\rm LT}^{-j})}}[rr]\ar[dl]_-{\mathrm{pr}_{L_m\mathfrak{F}_\infty/L_m}^\varpi}
      &  & L_m \otimes_{o_L} S_{\mathfrak{F}_\infty/L} \otimes_L D_{\cris,L}(V(\chi_\LT^{-j}))  \ar[dl]^{(\varpi^{-1})_*}             \\
   H^1(L_m,V(\chi_\mathrm{\rm LT}^{-j})(\varpi)) \ar[rr]^-{\exp_{L_m,V(\chi_\LT^{-j})(\varpi)}^*}
      &  &  L_m \otimes_L D_{\cris,L}(V(\chi_\LT^{-j})(\varpi))      }}
\end{equation}

The left face commutes by functoriality of corestriction maps, the top face by \eqref{f:descentXtoL}, the right face by \eqref{f:relativedR}, the back side by Theorem \ref{thm:bigColmez}  (respectively \eqref{f:interpolationunramified}), the lower face by  \eqref{f:dualExpDescent} (assuming that    ) and the front side by Theorem \ref{thm:ColmzezImproved}, at least on the image $I_V$, which is aimed to be indicated by the dotted arrows.

\begin{remark}
The cokernel of $\mathrm{pr}_{X_\infty/L_\infty},$ i.e., $H^1_{\mathrm{Iw}}(L_\infty/L,V)/I_V$, is controlled via the following convergent cohomological spectral sequence arising from \cite[Prop.\ 1.6.5 (3)]{FK}
\begin{equation*}
  E_2^{i,j}:=\mathrm{Tor}_{-i}^{\Lambda(X_\infty/L)}( H^j_{\mathrm{Iw}}(X_\infty/L,V),\Lambda(L_\infty/L)) \Rightarrow  H^{i+j}_{\mathrm{Iw}}(L_\infty/L,V),
\end{equation*}
which induces the following exact sequence of terms in lower degree
\begin{multline}
\mathrm{Tor}_{2}^{\Lambda(X_\infty/L)}( H^2_{\mathrm{Iw}}(X_\infty/L,V),\Lambda(L_\infty/L)) \to\\
 H^1_{\mathrm{Iw}}(X_\infty/L,V)_\Upsilon \to H^1_{\mathrm{Iw}}(L_\infty/L,V) \to \mathrm{Tor}_{1}^{\Lambda(X_\infty/L)}( H^2_{\mathrm{Iw}}(X_\infty/L,V),\Lambda(L_\infty/L))\to 0,
\end{multline}
where  we write $\Lambda(K/L)$ for the Iwasawa algebra $o_L[[G(K/L)]].$
The last non-trivial term identifies with $H^2_{\mathrm{Iw}}(X_\infty/L,V)^\Upsilon,$ in case $\Upsilon\cong\zp$ (and probably in general); in particular, the cokernel is a finite dimensional $L$-vector space. By local Tate duality, it is even zeroe, if $H^0(L_\infty,V^*(1))$ vanishes.
\end{remark}

\section{Comparison of regulator maps for different Lubin-Tate formal groups}\label{sec:comparison}

In this section we only consider $F=L=\mathbb{Q}_p$ and compare the regulator map constructed by Loeffler and Zerbes \cite{LZ}, i.e., the one attached to the multiplicative  formal group $\hat{\mathbb{G}}_m$ with prime element $p,$ to Kupferer's regulator map for a Lubin-Tate formal group over $\mathbb{Q}_p$ attached to a prime element
$$\pi=\epsilon p$$
for some $\epsilon\in\mathbb{Z}_p^\times.$ Both regulator maps take values (after base extension to $\mathbb{C}_p$ at least)   in
$$D_{\Qp}(\GaL,\mathbb{C}_p) \hten_{\mathbb{C}_p} D_{\Qp}(\Upsilon,\mathbb{C}_p)
	 \otimes_F  D_{\cris,F}(V(\tau^{-1}))\cong D_{\Qp}(G,\mathbb{C}_p)
	 \otimes_F  D_{\cris,F}(V(\tau^{-1}))$$
(but for different $\tau$!). As we are going to compare their interpolation properties we first analyse in \S \ref{sec zariski} to which extent the images under the regulator map are uniquely determined by them.

\subsection{On Zariski dense subsets of character varieties}\label{sec zariski}

For simplicity we assume $p\neq 2$. Let $A$ be a finite abelian group of order $a=|A|$ prime to $p$ and let in this subsection $\mathcal{G}=A\times \zp.$ We assume that $L$ is a finite extension of $\Qp$ containing $\mu_{a}$. The character variety $\widehat{\mathcal{G}}$ over $\Qp$ is
\[\widehat{\mathcal{G}}=\coprod_{\widehat{A}} \mathbb{B}_1,\]
where $\widehat{A}$ denotes the group of characters of $A$ with values in $\overline \QQ_p$ (they actually lie all in $L$ already) and $\mathbb{B}_1 $ is the open unit disk centered around $1$. For $L\subseteq K\subseteq \mathbb{C}_p$ any complete field, its $K$-valued points
\[\widehat{\mathcal{G}}(K)=\coprod_{\widehat{A}} \mathbb{B}_1(K)\]
correspond to continuous characters $\chi:\mathcal{G}\to K^\times$: if $z\in \mathbb{B}_1(K)$ lies in the component of $\omega\in \widehat{A},$ then $\chi_{\omega, z}(g):= \omega(\mathrm{pr}_A(g))z^{\mathrm{pr}_{\zp}(g)}.$

As a finite disjoint sum of open unit disks $\mathfrak{X}=\widehat{\mathcal{G}}$ is a smooth one dimensional quasi-Stein rigid analytic variety over $\Qp$, in particular it has an admissible covering $\mathfrak{X}=\bigcup_n \mathfrak{X}_n$ by an increasing sequence $\mathfrak{X}_1\subseteq \ldots \subseteq \mathfrak{X}_n\subseteq \ldots$ of affinoid subdomains such that the restriction maps have dense image.

\begin{example}\
\begin{enumerate}
  \item After choosing a primitive $(p-1)$th root of unity $\zeta_{p-1}$ and a topological generator $\kappa$ of $1+p\zp$ and thereby identifying $\mathbb{Z}_p^\times \cong \mathbb{Z}/p-1 \times \zp$ the powers  $\chi_{\mathrm{cyc}}^j$ of the cyclotomic character give rise to the characters $\chi_{\omega^j,\kappa^j},$ where $\omega(\bar{m})=\zeta_{p-1}^m.$
  \item Similarly, the powers $\chi_\mathrm{LT}^j$ correspond to points in the character variety $\widehat{\mathbb{Z}_p^\times}.$ Note that $\chi_{\LT}=\chi_\epsilon\chi_\mathrm{cyc}$ where $\chi_\epsilon$ is the unique continuous unramified character with $\chi_\epsilon(\mathrm{Frob}_p)=\epsilon,$ see \cite[Prop.\ 80]{stix}.
  \item Also $\Upsilon$ is of the form $\mathcal{G}$.
\end{enumerate}
\end{example}

We have to assume later that $\chi_\mathrm{\epsilon}$ factorizes over $\mathrm{Gal}(X_\infty/\Qp).$

Recall (cf. \cite[Chapter 3]{bellaichechenevier}) that a subset $Z$ of a rigid analytic space $X$ is called \textbf{Zariski-dense} if the only reduced analytic subset containing $Z$ is $X_{\rm red}.$ Here, analytic subset means the zero locus of a coherent ideal sheaf, see \cite[9.5.2, 9.5.3]{BGR}.\footnote{\label{footnote:Stein} For a reduced Stein space this is equivalent to requiring that an analytic function vanishing along $Z$ is identically zero.
This follows immediately from the definition and the fact that for a Stein space any coherent ideal sheaf is generated by global sections.}

An illustrating example is the set $(1+p)^\mathbb{N} \subset \mathbb{B}_{1}.$ It is Zariski dense because a function vanishing on this set has infinitely many zeroes inside the affinoids $\mathbb{B}_1(r)$, the closed ball centered at $1$ of radius  $r$, and thus vanishes identically along an admissible cover by the Weierstrass Preparation Theorem \cite[Thm.\ 5.2.2/1]{BGR}.

\begin{lemma}
Let $S_{Y_i}\subseteq Y_i$ be Zariski-dense subsets of affinoid rigid analytic spaces over $\mathbb{C}_p$, $i=1,2.$ Then also
\[S_{Y_1}\times S_{Y_2} \subseteq Y_1 \times Y_2\]
is a Zariski-dense subset.
\end{lemma}

\begin{proof}
Let $f\in \mathcal{O}_{\mathbb{C}_p}(Y_1 \times Y_2)$ such that $S_{Y_1}\times S_{Y_2} \subseteq V(f)$, the set of zeroes of $f$ and fix an element $y\in S_{Y_1} .$ Then $f(y,-)$ can be interpreted as element of $\mathcal{O}_{ \mathbb{C}_p}( Y_2)$ satisfying $S_{Y_2}\subseteq V(f(y,-)).$ Since by assumption $S_{Y_2}$ is Zariski-dense in $Y_2$ it follows that $f(y,-)\equiv 0$ on $Y_2.$ In other words, we have shown that $S_{Y_1}\times Y_2\subseteq V(f).$ Now, fix $z\in Y_2$ and consider $f(-,z)$ as function in $\mathcal{O}_{\mathbb{C}_p}(Y_1)$: it satisfies $S_{Y_1}\subseteq V(f(-,z)) $ by what he have just shown, whence $f(-,z)\equiv0$ on $Y_1$ due to the Zariski-density of $S_{Y_1}$ by assumption. It follows that $f\equiv 0$, whence the Zariski-closure of $S_{Y_1}\times S_{Y_2}$ is ${Y_1}\times {Y_2}$ as claimed.
\end{proof}

We want to apply this to the character variety $\mathfrak{Z}:=\widehat{G}$ of $G:$
\[\mathfrak{Z}= \mathfrak{X}\times \mathfrak{Y}.\]
with $\mathfrak{X}:=\widehat{\Upsilon}$ and $\mathfrak{Y}:=\widehat{\Gamma}$. Note that $\mathfrak{Z}$ is again a quasi-Stein space with admissible increasing covering $\mathfrak{Z}_n:=\mathfrak{X}_n\times \mathfrak{Y}_n.$

Suppose that $S_\mathfrak{X}$ is a subset of $\mathfrak{X}$   such that the intersection with any connected component of $\mathfrak{X}_n$ for all $n\geq 1$ is infinite, and analogously for $S_\mathfrak{Y}.$

\begin{lemma}\label{lem:uniques}
If   $f\in \mathcal{O}_{\mathbb{C}_p}(\mathfrak{Z})$ vanishes on $S_\mathfrak{X}\times S_\mathfrak{Y}$, then $f\equiv 0.$\footnote{Using footnote \ref{footnote:Stein} and the fact that the fibre product $\mathfrak{Z}= \mathfrak{X}\times \mathfrak{Y}$ of the reduced spaces $\mathfrak{X} $ and $\mathfrak{Y} $ is again reduced (compare with \cite[Lem.\ A.26.]{Stein}), one can show that this is equivalent to: The subset $S_\mathfrak{X}\times S_\mathfrak{Y}\subseteq \mathfrak{X}\times \mathfrak{Y}$ is Zariski-dense.}
\end{lemma}

\begin{proof}
We just have to show that $f$ restricted to  $\mathfrak{Z}_n$ vanishes for all $n\geq 1.$ By the proof of \cite[Lemma 1.1]{BSX} the intersection  of $S_\mathfrak{X}$ with any connected component of $\mathfrak{X}_n$ (being infinite by assumption) is Zariski dense in that component as the divisor of an analytic function of that component has finite support. The same holds analogously for $\mathfrak{Y},$ whence we are in the situation of the previous lemma, from which the claim follows.
\end{proof}

\begin{example}\label{ex:denseSets}
When working with the formal multiplicative group, we can take $S_\mathfrak{Y}:=\{\chi_\mathrm{cyc}^j\mid  j\geq m\}$ for any natural number $m.$ For $S_\mathfrak{X}$ one can choose $\widehat{A}\times \{z^j \mid j\geq n\}$ for any natural number $n$ and any $z\in 1+p\zp.$
\end{example}

\subsection{Comparing the interpolation formulae}

We   assume henceforth  that $\chi_\mathrm{\epsilon}$ factorizes over $\mathrm{Gal}(X_\infty/\Qp)$ which is equivalent to the condition that the composites  $L_\infty\mathfrak{F}_\infty$ and $\mathbb{Q}_p(\mu_{p^\infty})\mathfrak{F}_\infty$ coincide. We want to compare the regulator maps (for the group $G$) for the prime elements $p$ and $\pi=p\epsilon,$ the first one being  the cyclotomic one;  we omit the superscript $G$ and replace it by the prime element:

\[ \begin{xy} \xymatrix{
	 \mathbf{L}^p_{V}\colon \HIW^1(X_\infty/{\Qp},T) \ar[r] & D_{\Qp}(G,K)  	 \otimes_{\Qp}  D_{\cris,{\Qp}}(V ),
	 } \end{xy} \]

\[ \begin{xy} \xymatrix{
	 \mathbf{L}^\pi_{V}\colon \HIW^1(X_\infty/{\Qp},T) \ar[r] & D_{\Qp}(G,K) 	 \otimes_{\Qp}  D_{\cris,{\Qp}}(V(\tau^{-1})).
	 } \end{xy} \]
Here we assume that $K$ contains $\widehat{\mathfrak{F}_\infty}. $ In order to have the same targets we use the inverse of  $\mathbb{T}_{\tau^{-1}}$ from \eqref{f:dualexpcomp}. Note that this amounts to multiplication/tensoring by
$$\Omega^{-1}\otimes (\eta_\cyc\otimes\eta ^{\otimes (-1)})\in  D_{\cris,{\Qp}}(\Qp(\tau ))$$
%because we have $\Omega t_\LT=t_{\Qp}$
because of the following lemma.

\begin{lemma}\label{lem period}
We have $\Omega t_\LT=t_{\Qp}$.
\end{lemma}

\begin{proof}
Let $\eta' \in T_\pi'$ be as in \S \ref{sec:notation} and $\alpha_{\eta'}:\LT \to \widehat \GG_m$ the corresponding isomorphism of formal groups.
Then it is easy to see that
%the following commutative diagram
\begin{equation}\label{f:log}
\Omega \log_\LT(Z) = \log_{\widehat \GG_m}(\alpha_{\eta'}(Z)).
%   \xymatrix{
%     \LT  (\mathbb{C}_p)\ar[d]_{\log_{\LT }} \ar[r]^{\alpha_{\eta'}} & \widehat \GG_m(\mathbb{C}_p) \ar[d]^{\log_{\widehat \GG_m}} \\
%     \mathbb{C}_p \ar[r]_{\Omega\cdot} & \mathbb{C}_p,   }
\end{equation}
%in which $t(\alpha_{\eta'})$ is just multiplication by $\Omega_{\eta'} $.
Let $\omega_{\LT }$ and $\omega_{\widehat \GG_m}$ denote the elements in $\tilde{\mathbf{A}}^+:=W(\tilde{\mathbf{E}}^+)$ defined via \cite[Lem.\ 1.2]{KR} for ${\LT }$ and ${\widehat \GG_m}$, respectively. Then $t_{\LT }:=\log_{\LT }(\omega_{\LT })$ and $t_{\QQ_p}:=\log_{\widehat \GG_m}(\omega_{\widehat \GG_m})$ converge in $\tilde{\mathbf{B}}_\mathrm{rig}^+$, the Fr\'{e}chet completion of $ W(\tilde{\mathbf{E}}^+)[\frac{1}{p}],$ a subring of $\mathbf{B}_\mathrm{dR}^+.$ It follows from \eqref{f:log} and \cite[Prop.\ 2.1]{S} - saying   $ \alpha_{\eta'}(\omega_{\LT })=\omega_{\widehat \GG_m}$  -  that $  \Omega t_{\LT }= t_{\QQ_p}$.
%It is easy to check that $t_{\QQ_p}$ equals Fontaine's period $t$ in $p$-adic Hodge theory corresponding to the choice $\eta^\mathrm{cyc}.$
\end{proof}

%\cite[(2.1.5) in Remark 2.1.1]{SaV26}.

Recall that  $e_j^\pi:=t_{\LT}^{-j}\otimes \eta^{\otimes j}\in D_{\cris,\Qp}(\Qp(\chi_{\LT}^j))$ and $e^p_j:=t_{\Qp}^{-j}\otimes \eta_\cyc^{\otimes j}\in D_{\cris,\Qp}(\Qp(\chi_{\cyc}^j)).$

\begin{theorem}\label{thm:comparison}
   The maps $\mathbf{L}^p_{V}$ and $\id\otimes (\mathbb{T}_{\tau^{-1}})^{-1}\circ\mathbf{L}^\pi_{V}$ coincide.
\end{theorem}

\begin{proof}
Fix $y\in \HIW^1(X_\infty/F,T).$ Then, by Theorem \ref{thm:interpolationRegKu} evaluation of $\mathbf{L}^p_{V}(y)$ at $\rho=\varpi\chi_{\cyc}^j$ for $j\geq 0$ gives
\begin{align}\label{f:evalcyc}
   &  j!
    (1-p^{-1}\varphi_L^{-1})^{-1}(1- \varphi_L)\Big(\exp^*_{L,V(\varpi^{-1}\chi_{\cyc}^{-j}),\id}(y_{{\rho}^{-1}})\Big)\otimes e_{ j}^p
\end{align}
while evaluation of $\mathbf{L}_V^\pi(y)$ at the same character equals
\begin{align}\label{f:evalLT}
\notag  \mathbf{L}_V^\pi(y)(\varpi\chi_{\cyc}^j) & =  \mathbf{L}_V^\pi(y)(\varpi\chi_\epsilon^{-j}\chi_{\LT}^j) \\
   & =\Omega^{-j}  j!
    (1-\pi^{-1}\varphi_L^{-1})^{-1}(1-\frac{\pi^{}}{p}\varphi_L)\Big(\widetilde{\exp}^*_{L,V(\varpi^{-1}\chi_\epsilon^j\chi_{\LT}^{-j}),\id}(y_{{\rho}^{-1}})\Big)\otimes e_{ j}^\pi,\\
  \notag   & =\Omega^{-j}  j!
    (1-(p\epsilon)^{-1}\varphi_L^{-1})^{-1}(1-\frac{p\epsilon}{p}\varphi_L)\Big(\mathbb{T}_{\tau^{-1}}\circ{\exp}^*_{L,V(\varpi^{-1}\chi_{\cyc}^{-j}),\id}(y_{{\rho}^{-1}})\Big)\otimes e_{ j}^\pi,\\
  \notag  & = \mathbb{T}_{\tau^{-1}}\Big( j!
    (1-p^{-1}\varphi_L^{-1})^{-1}(1-\varphi_L)\Big({\exp}^*_{L,V(\varpi^{-1}\chi_{\cyc}^{-j}),\id}(y_{{\rho}^{-1}})\Big)\otimes \Omega^{-j}e_{ j}^\pi\Big).
\end{align}
Upon noting that under the identifications as in \eqref{f:Dcrisunramifiedtwist} $ \Omega^{-j}e_{ j}^\pi$ identifies with $e_{ j}^p$ we see that \eqref{f:evalcyc} and \eqref{f:evalLT} coincide after applying $\id\otimes (\mathbb{T}_{\tau^{-1}})^{-1}$ to the latter. Therefore the claim follows from Lemma \ref{lem:uniques} in combination with Example \ref{ex:denseSets}.
\end{proof}

\begin{acknowledgement}
We are very grateful to David Loeffler for suggesting the use  of big exponential maps for unramified towers in order to improve Colmez' reciprocity law in the Lubin-Tate setting from \cite{SanoV}.
Also we thank Rustam Steingart for discussions about Zariski density.
The project was funded by the Deutsche Forschungsgemeinschaft (DFG, German Research Foundation) under   TRR 326, {\it Geometry and Arithmetic of Uniformized Structures}, project-ID 444845124, and JSPS KAKENHI Grant Number 22K13896.
\end{acknowledgement}

%\noindent
%Otmar Venjakob\\
%Universit\"{a}t Heidelberg,  Mathematisches Institut,\\  Im Neuenheimer Feld 288,  69120
%Heidelberg,  Germany,\\
% http://www.mathi.uni-heidelberg.de/$\,\tilde{}\,$venjakob/\\
%venjakob@mathi.uni-heidelberg.de

%\printendnotes[custom]

\end{document}